\documentclass{amsart}

\newtheorem{theorem}{Theorem}[section]
\newtheorem{lemma}[theorem]{Lemma}
\newtheorem{corollary}[theorem]{Corollary}

\theoremstyle{definition}
\newtheorem{definition}[theorem]{Definition}
\newtheorem{example}[theorem]{Example}

\newtheorem{proposition}[theorem]{Proposition}

\theoremstyle{remark}
\newtheorem{remark}[theorem]{Remark}

\numberwithin{equation}{section}

\def \R {\bf R}

\newcommand\Ac{\mathcal A}

\def\R{\mathcal R}

\def\tr{{\rm tr}}

\def\Jc{\mathcal J}

\def\ee{\'{e}}

\def\Bc{{\mathcal B}}
\def\gm{{g_0}}

\def\Cc{{\cal C}}

\def\Sc{{\mathcal S}}

\def\Tc{{\mathcal T}}

\newcommand\Rc{{\mathcal R}}
\newcommand\Lc{{\mathcal L}}
\newcommand\Fc{{\mathcal F}}

\newcommand\Uc{{\mathcal U}}

\def\Mc{{\mathcal M}}

\newcommand\Hc{{\mathcal H}}
\newcommand\Dc{{\mathcal D}}

\def\Cc{\mathcal{C}}

\def\im{\textrm{Im}}
\def\bea{\begin{eqnarray}}
\def\eea{\end{eqnarray}}
\def\be{\begin{equation}}
\def\ee{\end{equation}}

\def\Cc{\mathcal{C}}

\def\Cc{\mathcal{C}}

\def\Rcc{\stackrel{\circ}{R}}

\def\Cc{\mathcal{C}}

\def\im{\textrm{Im}}

\def\<{\langle}
\def\>{\rangle}

\def\im{\textrm{Im}}

\begin{document}

\title[Deformation and rigidity results]{Deformation and rigidity results for the  $2k$-Ricci tensor and the $2k$-Gauss-Bonnet curvature}

\author[T. Ca\'ula]{Tiago Ca\'ula}
\address{Unilab,
Campus da Liberdade,
Avenida da Aboli\c{c}\~ao, 3, Centro, Reden\c{c}\~ao/CE, Brazil - CEP: 62.790-000}
\email{ticaula@gmail.com}

\author[L. L. de Lima]{Levi Lopes de Lima}
\address{Federal University of Cear\'a, Department of Mathematics, Campus do Pici, Av. Humberto Monte, s/n, Bloco 914, Fortaleza/CE, Brazil - CEP: 60455-900.}
\email{levi@math.ufc.br}

\author[N. L. Santos]{Newton Luis Santos}
\address{
Federal University of Piau\'{\i}, Department of Mathematics, Campus Petronio Portela,
Ininga, Teresina/PI, Brazil - CEP:  64049-550 }
\email{newtonls@ufpi.br}

%
%
%

\begin{abstract}
We present several deformation and rigidity results within the classes of closed Riemannian manifolds which either are $2k$-Einstein (in the sense that their $2k$-Ricci tensor is constant) or have constant $2k$-Gauss-Bonnet curvature. The results hold for a family of manifolds containing all non-flat space forms and the main ingredients in the  proofs are explicit formulae for the linearizations of the above invariants obtained by means of the formalism of double forms.
\end{abstract}

\maketitle

\section{Introduction}

In Riemannian Geometry it is natural to consider invariants constructed out of the curvature tensor  by
means of natural algebraic constructions (such as tensor products followed by  contractions), the simplest of these being of course the Ricci tensor and the scalar curvature. From this point of view, the problem of determining the Riemannian structures with the property that one such invariant is {\em constant} in a suitable sense stands out by its evident naturality. The purpose of this paper is precisely to present some deformation and rigidity results in the context of some Riemannian invariants with the above kind of structure.

The class of invariants considered here depend on the curvature in a polynomial fashion and are in a sense the simplest such examples. In fact, they can  be described by means  of the concept of
{\em double form}, which is simply an element of the bi-graded algebra
\[
\Ac^{\bullet,\bullet}(X)=\Ac^\bullet(X)\otimes_{\Dc(X)}\Ac^\bullet(X),
\]
where $X$ is a smooth manifold, $\Dc(X)$ is the ring of smooth functions on $X$ and
\[
\Ac^\bullet(X)=\oplus_{r\geq 0}\Ac^r(X)
\]
is the graded $\Dc(X)$-algebra of differential forms; see Section \ref{2kriccgb} for further details.
Notice that if  $g$ is a Riemannian metric on  $X$ then  $g\in \Ac^{1,1}(X)$ and its curvature tensor $R_g\in\Ac^{2,2}(X)$.
Moreover, there exists a natural {contraction operator} $c_g$, which is actually the (pointwise) adjoint of multiplication by $g$ with respect to the inner product on forms.
With this terminology at hand, the Ricci tensor of $g$ can be expressed as
\[
{\rm Ric}_g=c_gR_g,
\]
so it is natural to consider, for $k\geq 1$, the $2k$-{\em Ricci tensor} given by
\begin{equation}\label{simb}
\Rc^{(2k)}_g=c^{2k-1}_gR^k_g,
\end{equation}
an element of $\Ac^{1,1}(X)$ which is symmetric in its entries.
In this context, a metric $g$ is $2k$-{\em Einstein} if it satisfies  \begin{equation}\label{einsteincond}
\Rc^{(2k)}_g=\lambda g,
\end{equation}
for some constant $\lambda$.
One should remark that, similarly to what happens in the Einstein case ($k=1$), this condition also admits a variational interpretation;
see Proposition \ref{equivalenteinstein} below.

The study of Einstein metrics is a honorable topic in Riemannian Geometry; see
\cite{Be} for a comprehensive introduction to the subject. In particular, it is well-known
that the corresponding moduli space
always appears, for a {\em closed} smooth manifold $X$, in finite dimensional families. Moreover, in some cases it is verified that
such structures display local rigidity phenomena;  see \cite{Be} for a survey of such results. This is the case, for instance, if  $(X,g)$ is a spherical space form, which can be seen as an extension
of a   famous rigidity result due to Calabi \cite{Ca}, or  if $(X,g)$ is a hyperbolic space form, which is a local generalization of
a remarkable rigidity theorem by
Mostow \cite{Mo}.

The case $k\geq 2$, however, is a bit more complicated, essentially due to the fact that in general $\Rc^{(2k)}_g$ is homogeneous of degree $k$ in $R_g$ (or, equivalently, in the second order derivatives of $g$), which implies  that the principal symbol  of the linearization of (\ref{simb}) depends on $R_g$ for $k\geq 2$. This should be compared with the case $k=1$ mentioned above, where  this symbol depends only on the derivatives of $g$ up to first order, so that the corresponding linearization is always elliptic in a suitable {gauge}; in fact, this is precisely the information that leads to the finiteness result mentioned above. But for $k\geq 2$ the linearization is not elliptic in general and the question of
exhibiting examples of $2k$-Einstein where ellipticity  (with the consequent local finiteness of the dimension of the moduli space) is restored, acquires fundamental relevance.

In this work we single out a class of
Riemannian manifolds for which this program may be carried out in a satisfactory manner. More precisely, if $n\geq 5$  and $2\leq k< n/2$, let us denote by $\Hc_{n,k}$ the class of closed Riemanian manifolds $(X,g)$ which are  $2k$-Einstein and additionally meet the curvature condition
\begin{equation}\label{condcurva}
R_g^{k-1}=\mu_kg^{2k-2},\quad \mu_k\neq 0,
\end{equation}
which actually means that the manifold in question has {\em constant} $(2k-2)$-sectional curvature in the sense of Thorpe; see Proposition \ref{lipkill}. In particular,
$\Hc_{n,k}$ contains all space forms except the flat ones (i.e. those satisfying $R_g=0$). Our first result (Theorem  \ref{main}) says that if $(X,g)\in \Hc_{n,k}$ then the corresponding moduli space of $2k$-Einstein structures, denoted $\mathfrak E^{(2k)}(X)$, is finite dimensional at $\langle g\rangle$, the class of $g$. This follows from the fact that the linearization of (\ref{simb}) in $(X,g)$ is elliptic in a suitable gauge. Actually, a much more precise result concerning the local structure of $\mathfrak E^{(2k)}(X)$ around $\langle g\rangle$ is obtained in Theorem \ref{main4}. Also, this information is complemented with Theorem \ref{main2}, which provides examples of manifolds in
$\Hc_{n,k}$ which are rigid, as $2k$-Einstein structures,
under a certain assumption on the eigenvalues of $\Rcc_g$, the natural action of $R_g$ on $\Ac^{1,1}(X)$. This subclass of examples includes in particular all {\em nonflat} space forms (Corollary \ref{constcurv2}), so that the classical rigidity results mentioned above are shown to admit extensions to the
$2k$-Einstein context.

\begin{remark}\label{gene}
For spherical space forms, the above results were previously proved in \cite{dLS1}. Also, it is proved in Proposition \ref{equivform} below that, in the presence of (\ref{condcurva}), being $2k$-Einstein is equivalent to being Einstein. Thus, the results above can be seen as generalizations of those established  in \cite{dLS1} to Einstein manifolds with constant $(2k-2)$-sectional curvature.
\end{remark}

A further contraction of  $\Rc^{(2k)}$ yields the so-called $2k$-{\em Gauss-Bonnet curvature},
\begin{equation}\label{gb}
\Sc^{(2k)}_g=\frac{1}{(2k)!}{c}_g\Rc^{(2k)}_g,
\end{equation}
a scalar invariant of $g$ which is homogeneous of degree $k$ in its second order derivatives. These invariants
are notably ubiquitous in Differential Geometry, appearing for instance in
Weyl's expression for the volume of tubes
\cite{G} and Chern's kinematic formulae \cite{Ch} for quermassintegrals. Notice that  $\Sc^{(2)}_g=\kappa_g/2$, where $\kappa_g=c_g{\rm Ric}_g$ is the scalar curvature of $g$. Now, classically, the scalar curvature plays a fundamental role  in Conformal Geometry, in connection with the famous Yamabe problem \cite{LP}. It is thus natural to formulate the corresponding  problem for the Gauss-Bonnet curvatures: given a metric $g$ in $X$, is there $g'$ conformal to $g$ so that $\Sc^{(2k)}_{g'}$ is constant? This problem, which admits a nice variational characterization (Proposition \ref{characgbconst}), has been considered so far in case $g$ is locally conformally flat \cite{LL} \cite{GW}, for it is then equivalent to the $\sigma_k$-Yamabe problem (Proposition \ref{doisyam}). In this work we present new examples of manifolds $(X,g)$ with non-null Weyl tensor for which the problem has a positive solution. More precisely, if $\Hc'_{n,k}$ represents the subclass of manifolds  $(X,g)$ in $\Hc_{n,k}$ isometrically distinct from round spheres, then it is shown in Theorem \ref{main5}) that any metric sufficiently close to $g$
is conformally equivalent to a metric with constant $2k$-Gauss-Bonnet curvature. We remark that the corresponding result for space forms has been previously verified in \cite{dLS2}.

This paper is organized as follows. In Section \ref{2kriccgb} we review the basic properties of the Riemannian invariants mentioned above. The deformation and rigidity results for $2k$-Einstein structures are stated in Section \ref{def2keinstein} and proved in Section \ref{demresrig}. This uses the expression for the linearization of (\ref{simb}) obtained in Section \ref{lineariricc}. Finally, the deformation result (the local Yamabe problem) for the Gauss-Bonnet curvatures is presented in Section \ref{yamabegb} and proved in Section \ref{demmain5}. Again, this uses an expression for the linearization of $\Sc^{(2k)}_g$ obtained in Section \ref{lin}.

\section{The $2k$-Ricci tensors and the Gauss-Bonnet curvatures}
\label{2kriccgb}

In this section we review the definition of an array of Riemannian invariants
that generalize the Ricci tensor ${\rm Ric}_g$, the scalar curvature $\kappa_g$ and the Einstein tensor
\begin{equation}\label{einstt}
E_g={\rm Ric}_g-\frac{\kappa_g}{2}g
\end{equation}
of a Riemannian manifold $(X,g)$ and collect their basic properties. As always, we assume that $X$ is closed. Also, we adhere to the sign conventions of \cite{Be}.

In what follows, $\Tc^{(r,s)}(X)=\Gamma(\otimes^{(r,s)}X)$ is the space of smooth tensors of type $(r,s)$, so that $S^r(X)\subset \Tc^{(r,0)}(X)$ will denote the space of symmetric {\em covariant} tensors of degree $r$ and $\Ac^r(X)\subset \Tc^{(r,0)}(X)$ is the space of differential $r$-forms. More generally, if $\mathcal E$ is a metric vector bundle over $X$ endowed with a compatible connection, we represent by $\Ac^{r}(X;\mathcal E)$ the space of $\mathcal E$-valued $r$-forms over $X$. We also recall the {\em divergence operator}
${\delta_g}:S^r(X)\to S^{r-1}(X)$ given by
\[
({\delta_g} T)_{i_2\cdots i_{r}}= -g^{ij}\nabla_i T_{ji_2\cdots i_{r}}=\nabla_iT^i_{i_2\cdots i_{r}},
\]
where
$\nabla_iT=\nabla_{\partial_i}T$ is the covariant derivative.
We also remark that twice contraction of the differential Bianchi identity yields
\[
\delta_g{\rm Ric}_g+\frac{d\kappa_g}{2}=0,
\]
or  equivalently, the Einstein tensor is divergence free:
\begin{equation}\label{divinull}
\delta_gE_g=0.
\end{equation}
Our aim now is to point out the existence of a natural family of  divergence free tensors
$L^{(2k)}_g\in S^2(X)$, $1\leq k\leq [(n-1)/2]$, the so-called {\em Lovelock tensors}, which generalize the Einstein tensor in the sense that $L^{(2)}_g$ is proportional to $E_g$.

Recall that given a vector field $z\in \mathcal X(X):=\Tc^{(0,1)}(X)$ and a local volume element $\Omega$, we have
\[
d i_z\Omega=d i_z\Omega+ i_zd\Omega=\Lc_z\Omega=({\rm div}\,z)\Omega,
\]
where $i_z$ is contraction with $z$ and  $\Lc_z$ is Lie derivative. In this way, the correspondence $z\leftrightarrow \omega=i_z\Omega$ defines an isomorphism  between  $T^*X=TX$ and $\Lambda^{n-1}(X)$ so that  $\delta_g z=0$ if and only if $d\omega=0$. Similarly, the  correspondence
$
z_1\otimes z_2 \leftrightarrow i_{z_1}\Omega\otimes i_{z_2}\Omega
$
defines an  isomorphism between  $\otimes^{(0,2)}(X)$ and $\Lambda^{n-1}(X)\otimes \Lambda^{n-1}(X)$, the bundle of  $(n-1)$-forms  taking values on $(n-1)$-forms, which is well defined even if $X$ is not orientable.
Moreover, if ${\rm Sym}^2(X)\subset \otimes^{(0,2)}(X)$ is the space of {\em symmetric} $(0,2)$-tensors then we get an isomorphism between ${\rm Sym}^2(X)$ and ${\rm Sym}^2(\Lambda^{n-1}(X))$, where, by definition,  $\eta\in{\rm Sym}^2(\Lambda^{r}(X))\subset \Lambda^{r}(X)\otimes \Lambda^{r}(X)$ if and only if
\[
\eta(v_1\wedge\ldots\wedge v_{r}\otimes w_1\wedge\ldots\wedge
w_{r})=\eta(w_1\wedge\ldots\wedge w_{r}\otimes v_1\wedge\ldots\wedge
v_{r}).
\]
In what follows, we shall write
\[
S^2(\Lambda^r(X))=\Gamma({\rm Sym}^2(\Lambda^r(X))).
\]

A simple computation shows that
$T\in S^2(X)$ satisfies ${\delta_g}T=0$ if and only if the corresponding section $\eta\in S^2(\Lambda^{n-1}(X))\subset \Ac^{n-1}(X,\Lambda^{n-1}(X))$ meets $d^{\nabla}\eta=0$, where $d^\nabla$ is the standard exterior covariant derivative.
Noticing that $g\in S^2(X)=S^2(\Lambda^1(X))$ and $R\in S^2(\Lambda^{2}(X))$, let us define
\begin{equation}\label{lov}
\tilde L^{(2k)}_g= R_g{\wedge}\stackrel{k\, {\rm times}} {\cdots} \wedge R_g\wedge g\wedge \cdots \wedge g\in \Sc^2(\Lambda^{n-1}(X)),\quad 1\leq k\leq \left[\frac{n-1}{2}\right].
\end{equation}
Since $d^{\nabla}g=0$ (metric compatibility) and $d^\nabla R_g=0$ (Bianchi identity), we see that $d^\nabla \tilde L^{(2k)}_g=0$, so that the corresponding tensor $L^{(2k)}_g\in S^2(X)$ satisfies $\delta_gL^{(2k)}_g=0$. These are precisely the {\em  Lovelock tensors} \cite{Lo}.

\begin{example}\label{hyper}
Assume that $X^n\hookrightarrow\mathbb R^{n+1}$ isometrically and let $A\in S^2(X)$ be the corresponding Weingarten map. Then the Gauss equation says that $R=\frac{1}{2}A\wedge A$ and (\ref{lov}) becomes
\[
\tilde L^{(2k)}= \frac{1}{2^k}A{\wedge}\stackrel{2k} {\cdots} \wedge A\wedge g\wedge \cdots \wedge g.
\]
A computation shows that
\[
L^{(2k)}=c_{n,k}P_{2k},
\]
where $P_r=S_rI-P_{r-1}A$ is the {\em Newton tensor} of order $r$ and $S_r$ is the ${r}^{\rm th}$-elementary symmetric function in the eigenvalues of $A$. In particular, if $r=2$ we have $S_2=\kappa/2$ and $P_1A=S_1A-A^2={\rm Ric}$, so that $P_2=-E$.
\end{example}


\begin{remark}\label{span}
The Lovelock  tensors admit a local expansion of the form
\begin{equation}\label{local}
(L^{(2k)}_g)^i_j=d_{n,k}\delta^{ii_1i_2\ldots i_{2k-1}i_{2k}}_{jj_1j_2\ldots j_{2k-1}j_{2k}}R_{i_1i_2}^{j_1j_2}\ldots R_{i_{2k-1}i_{2k}}^{j_{2k-1}j_{2k}},
\end{equation}
where $d_{n,k}$ is a universal  constant, $\delta$ is the generalized Kronecker delta and $R^{ij}_{kl}$ are the coefficients of $R_g\in\Tc^{(2,2)}(X)$ with respect to a local orthonormal frame. Thus, $L^{(2)}_g$ is  proportional to the Einstein tensor, as desired.
Moreover, given a metric $g$ in $X$, it is proved in \cite{Lo} that  the Lovelock tensors span the space of natural, second order and  divergence free elements (with respect to $g$) in $S^2(X)$; see \cite{NN} for a modern proof.
\end{remark}

It turns out that the above concepts can be reformulated in  terms of the notion of {\em double form}. Let us start by considering a smooth manifold $X$ of dimension $n\geq 3$ and recalling that
$\Ac^r(X)$ is a module over the ring $\Dc(X)$ of smooth functions defined on $X$.

\begin{definition}\label{doubleforms}
The space of {\em double forms} of bi-degree
$(r,s)$ is given by
\[
\Ac^{r,s}(X)=\Ac^r(X)\otimes_{\Dc(X)}\Ac^s(X).
\]
Equivalently,
\[
\Ac^{r,s}(X)=\Gamma(\Lambda^{r,s}(X)),
\]
where
\[
\Lambda^{r,s}(X)=\Lambda^r(X)\otimes \Lambda^s(X).
\]
\end{definition}

We also set
\[
\Ac^{\bullet,\bullet}(X)=\oplus_{r,s\geq 0}\Ac^{r,s}(X).
\]
We thus see  that $\Ac^{\bullet,\bullet}(X)$ is a bi-graded associative algebra,
the so-called  {\em algebra of double forms}.

For instance, any bilinear form on tangent vectors is a
$(1,1)$-form.
In particular, a Riemannian metric $g$ on $X$
is a $(1,1)$-form.
Moreover, the curvature tensor $R_g$ of $g$ can be seen as a
$(2,2)$-form. In fact, if we define ${\mathcal C}^r(X)\subset \Ac^{r,r}(X)$
as being the space of  $(r,r)$-forms satisfying  the symmetry condition
\[
\omega(x_1\wedge\ldots\wedge x_{r}\otimes y_1\wedge\ldots\wedge
y_{r})=\omega(y_1\wedge\ldots\wedge y_{r}\otimes x_1\wedge\ldots\wedge
x_{r}),
\]
then any bilinear form ($g$, in particular) lies in  ${\mathcal C}^1(X)$, and $R_g\in{\mathcal C}^2(X)$\footnote{In this notation, $\Cc^1(X)=S^2(X)$.}. Detailed accounts of the theory of double forms can be found in
\cite{L1}, \cite{L2} and \cite{G}.

Notice that multiplication by the metric defines a linear map
$g:\Ac^{r-1,s-1}(X)\to\Ac^{r,s}(X)$.
Also, the  {\em contraction operator} $c_g:\Ac^{r,s}(X)\to\Ac^{r-1,s-1}(X)$ is given by
\[
(c_g\omega)(x_1\wedge\ldots\wedge x_{r-1}\otimes y_1\wedge\ldots\wedge
y_{r-1})=\sum_i \omega(e_i\wedge x_1\wedge\ldots\wedge
x_{r-1}\otimes e_i\wedge y_1\wedge\ldots\wedge y_{r-1}),
\]
where $\{e_i\}$ is a local orthonormal frame. It is easily shown that
$g$ and $c_g$
are adjoints to each other with respect to the natural inner product defined
in
\[
\Lambda^{\bullet,\bullet}(X)_p=\oplus_{r,s\geq 0}\Lambda^{r,s}(X)_p, \quad p\in X.
\]
Moreover, these operators satisfy the following {commutation rule}, established in \cite{L2}:
for $\eta\in\Ac^{r,s}(X)$ there holds
\begin{eqnarray}\label{commutation}
\frac{1}{m!}c^l_gg^m\eta & = & \frac{1}{m!}g^mc^l_g\eta+\nonumber\\
& & \quad  +\sum_{q=1}^{\min \{l,m\}} C^l_q\prod_{i=0}^{q-1}(n-r-s+l-m-i)\frac{g^{m-q}}{(m-q)!}c^{l-q}_g\eta,
\end{eqnarray}
where $C^l_q$ is the usual binomial coefficient.
In particular, the following special case deserves some attention:
\begin{equation}\label{kulkarni}
c_gg\eta=gc_g\eta+(n-r-s)\eta,\qquad \eta\in\mathcal A^{r,s}(X).
\end{equation}

\begin{remark}\label{const}{
Using the language of double forms, that a Riemannian manifold  $(X,g)$ has {\em constant} sectional curvature $\mu\in\mathbb R$ is equivalent to the validity of the identity
$R_g=\frac{\mu}{2}g^2$.}
\end{remark}

The contraction operator can be used to rewrite the Ricci tensor and the scalar curvature of $(X,g)$ as  ${\rm Ric}_g=c_gR_g$ and ${\kappa}_g=c^2_gR_g$. This motivates the following definition.

\begin{definition}\label{rew}
For $1\leq k\leq n/2$ we define the $2k$-{\em Ricci tensor} and the $2k$-{\em  Gauss-Bonnet curvature}, respectively, by
\begin{equation}\label{def}
{\mathcal R}^{(2k)}_g=c^{2k-1}_gR^k_g,\quad {\mathcal S}^{(2k)}_g=\frac{1}{(2k)!}c^{2k}_gR^k_g.
\end{equation}
\end{definition}

Accordingly, it is now possible to rewrite the Lovelock tensor, up to a universal constant, as $L^{(2k)}_g=c'_{n,k}{\mathcal J}^{(2k)}_g$, where
\begin{equation}\label{locve}
{\mathcal J}^{(2k)}_g=\frac{{\mathcal R}^{(2k)}_g}{(2k-1)!}-{\mathcal S}^{(2k)}_g g.
\end{equation}
This emphasizes the similarity with the Einstein tensor in (\ref{einstt}).

The following definition plays a central role in this work.

\begin{definition}\label{2keinstein}
\cite{L2} We say that  $(X,g)$ is $2k$-{\em Einstein} if there exists a smooth function  $\lambda$ on $X$ such that
\begin{equation}\label{einsteincond}
{\mathcal R}_g^{(2k)}=\lambda g.
\end{equation}
\end{definition}

Thus, $2$-Einstein means precisely that  $(X,g)$ is  Einstein in the usual sense. We will see in  Proposition \ref{equivalenteinstein} that  if $X$ is closed then $2k$-Einstein metrics are critical points for the {\em Hilbert-Einstein-Lovelock functional} given by
\[
{\mathcal F}^{(2k)}(g)=\int_X {\mathcal S}^{(2k)}_g\nu_g,
\]
restricted to the space ${\mathcal M}_1(X)$ of unit volume metrics on $X$. Here, $\nu_g$ is the volume element of $g$.  In particular, examples of $2k$-Einstein manifolds include space forms and isotropically irreducible homogeneous manifolds \cite{Be}. Moreover, if $2k=n$ then {\em any} metric  on $X$ is $2k$-Einstein, since in this case ${\mathcal S}^{(n)}_g$ is, up to a constant, the Gauss-Bonnet integrand. Thus, we may assume from now on that
$n>2k$.

\begin{proposition}\label{constante}
If $n>2k$ and $(X,g)$ is $2k$-Einstein then $\lambda$ is constant. In particular,
$\Sc^{(2k)}_g$ is constant.
\end{proposition}

\begin{proof}
Notice that $\delta_g\Jc^{(2k)}_g=0$ means that
\[
\delta_g\Rc^{(2k)}+(2k-1)!d\Sc^{(2k)}=0,
\]
and combining this with (\ref{einsteincond}) we then see that
the function
\[
\mu=\lambda-(2k-1)!\Sc^{(2k)}
\]
is constant. On the other hand,
since
\begin{equation}\label{cont}
{\rm tr}_g {\mathcal R}^{(2k)}_g=\langle c^{2k-1}_gR^k_g,g\rangle=c^{2k}_gR^k_g=(2k)!{\mathcal S}^{(2k)}_g,
\end{equation}
we have, again by (\ref{einsteincond}),
\begin{equation}\label{const}
\lambda=\frac{(2k)!}{n}{\mathcal S}^{(2k)}_g,
\end{equation}
and the result follows.
\end{proof}

\begin{example}\label{loveexamples}
Examples of $2k$-Einstein manifolds appear as black hole solutions in Lovelock gravity \cite{CTZ}. For instance, the manifold $\mathbb R\times I\times \mathbb R^{n-1}$ with coordinates $(t,r,\theta)$, where $I\subset (0,+\infty)$ is an interval, carries such a metric, namely,
\[
g=\pm F(r)dt^2+g_0,
\]
where
\[
g_0=F(r)^{-1}dr^2+r^2d\Theta^2,
\]
$d\Theta^2$ is the round metric in $\mathbb S^{n-1}$ and
\[
F(r)=1+\epsilon r^2-2mr^{2-\frac{n}{k}}.
\]
Here, $m\in \mathbb R$ is the \lq total mass\rq\, of the solution and $\epsilon=0$ or $\epsilon=\pm 1$ (for a non-vanishing cosmological constant). We also note that the Riemannian metric $g_0$ on the space-like slice $t=0$ has {\em constant} $2k$-Gauss-Bonnet curvature. For $k=1$ we recover the so-called Schwarzschild-type solutions  of Einstein gravity.
\end{example}

The formalism of double forms can also be used to single out a class of Riemannian manifolds  that will play a central role in this work.

\begin{definition}\label{thorpe}\cite{T}
Given $k\geq 2$, we say that  $(X,g)$ {\em has} $(2k-2)$-{\em constant\, secctional\, curvature} if there exists $\mu_k\in\mathbb R$ such that
\begin{equation}\label{thorpe2}
R^{k-1}=\mu_k g^{2k-2}.
\end{equation}
\end{definition}

The case $k=2$ corresponds to space forms; see Remark \ref{const}. In general, the condition (\ref{thorpe2}) can be geometrically interpreted in the following way. Given a tangent $(2k-2)$-plane $\mathfrak p\subset T_pX$, $p\in X$, there exists a neighborhood $U\subset \mathfrak p$ containing the origin
such that $\exp_p U\subset X$ is an embedded submanifold which is totally geodesic at
$p$. In this way, we can associate to each $\mathfrak p$ the  $(2k-2)$-Gauss-Bonnet curvature
of $\exp_p U$ at $p$, which turns out to be an invariant of $(X,g)$ at $p$, termed the {\em $(2k-2)$-sectional\, curvature} of $X$ at $p$ in the direction of  $\mathfrak p$, and denoted by $K(p,\mathfrak p)$\footnote{In the literature, this invariant is also called the {\em Lipschitz-Killing curvature}.}.

\begin{proposition}\label{lipkill}\cite{T}
For a Riemannian manifold $(X,g)$, (\ref{thorpe2}) happens if and only if
$K(p,\mathfrak p)$ does {\em not} depend on the pair $(p,\mathfrak p)$.
\end{proposition}

We denote by  $\Mc(X)$ the set
of smooth Riemannian metrics on $X$ and by $\Mc_1(X)$ the subset of unit volume metrics.
With respect to the $C^\infty$ compact-open topology, $\Mc(X)$ is an open convex cone which has $\Mc_1(X)$ as a basis.
In particular, if $g\in \Mc(X)$ and $h\in S^2(X)$ then $g+th\in\Mc(X)$ if $t\in(-\epsilon,\epsilon)$ with $\epsilon>0$ small enough. In this way, if $g\mapsto B_g$ is a Riemannian invariant (taking values in some open subset of the space of sections of some vector bundle) it makes sense to define its  {\em linearization} at $g$ in the direction of $h$ by
\begin{equation}\label{lineari}
\dot B_gh=\lim_{t\to 0}\frac{B_{g+th}-B_g}{t}.
\end{equation}

The following proposition describes the well-known formulae for the linearizations of the Ricci tensor and the scalar curvature. For this we need to introduce the
{\em Lichnerowicz Laplacian}, $\Delta_L:S^2(X)\to S^2(X)$,
\begin{equation}\label{lich}
\Delta_Lh=\nabla^*\nabla h +{\rm Ric}_g\circ h+h\circ {\rm Ric}_g-2\!\Rcc_g\!\!h,
\end{equation}
where $\nabla^*\nabla$ is the Bochner Laplacian,
\begin{equation}
(h\circ k)(x,y)=\sum_{i=1}^n h(x,e_i)k(e_i,y), \label{jj}
\end{equation}
and
\begin{equation}\label{jjj}
(\Rcc_g\!\!h)(x,y)=\sum_{i=1}^nh(R_g(x,e_i)y,e_i),
\end{equation}
with $\{e_i\}$ being an orthonormal frame. We also need the
{\em Bianchi operator} $\beta_g:S^2(X)\to \Ac^1(X)$,
\begin{equation}\label{bianchidef}
\beta_gh=\delta_gh+\frac{1}{2}d{\rm tr}_gh.
\end{equation}

\begin{proposition}\label{linearcomp}
If $g\in \Mc(X)$ and $h\in S^2(X)$ then there holds
\begin{equation}\label{ricciline}
\dot{\rm Ric}_gh=\frac{1}{2}\left(\Delta_Lh-\Lc_{(\beta_gh)^{\sharp}}g\right),
\end{equation}
where $\omega^\sharp\in\mathcal X(X)$ is the vector field dual to  $\omega\in\Ac^1(X)$ and $\Lc$ is Lie derivative. Moreover,
\begin{equation}\label{curvscarvar}
\dot\kappa_gh=\Delta_g{\rm tr}_gh+\delta_g\delta_gh-\langle{\rm Ric}_g,h\rangle,
\end{equation}
where $\Delta_g$ is the metric Laplacian.
\end{proposition}

If $D(X)$ is the group of smooth diffeomorphisms of $X$, then there exists a natural action $\xi:(\mathbb R^+\times D(X))\times \Mc(X)\to \Mc(X) $,
\[
\xi((t,\phi),g)=t^2\phi^*g.
\]
Obviously, two metrics in an orbit of this action have the same geometric properties.
We can also consider the restricted action
$\xi_1:D(X)\times\Mc(X)\to\Mc(X)$, with $\xi_1=\xi|_{\{1\}\times D(X)}$.
We thus see that isometry classes of metrics correspond
to elements of $\Mc(X)/D(X)$ and {\em globally} homothetic classes of metrics correspond to elements of $\Mc_1(X)/D(X)$. The elements of $\Mc(X)/D(X)$ are
called {\em Riemannian structures}.

A basic problem in Riemannian  Geometry  consists of understanding the set of Riemannian structures in  a given closed manifold  $X$ satisfying some geometric condition (Einstein, $2k$-Einstein, constant Gauss-Bonnet curvature, etc.). With this goal in mind,
it is crucial to understand the structure of the orbit space for the above actions. In fact, here we only need the infinitesimal picture so
we start by noticing that, at least formally, the tangent space
to the orbit
\[
O(g)=\{\xi_1(\phi,g);\phi\in D(X)\},\quad g\in \Mc(X),
\]
is given by
\[
T_gO(g)=\{\Lc_{\omega^{\sharp}}g;\omega\in \Ac^1(X)\}.
\]
Thus, $T_gO(g)={\rm im}\,\delta^*_g$, where $\delta^*_g:\Ac^1(X)\to S^2(X)$ is given by
\[
\delta^*_g\omega=\frac{1}{2}\Lc_{\omega^\sharp}.
\]
The notation for $\delta^*_g$ is justified by the fact that this is the $L^2$ adjoint of $\delta_g:S^2(X)\to\Ac^1(X)$.

Locally, we have
\[
(\delta^*_g\omega)_{ij}=\frac{1}{2}(\nabla_i\omega_j+\nabla_j\omega_i),
\]
which implies that the principal symbol of $\delta^*_g$ is injective (outside of the zero section). It follows that  $\delta_g\delta^*_g:S^2(X)\to S^2(X)$ is elliptic and an argument due to Berger and Ebin \cite{BE} gives the decomposition
\begin{equation}\label{be}
S^2(X)={\rm im}\,\delta^*_g\oplus\ker  \delta_g,
\end{equation}
which is orthogonal with respect to the $L^2$  inner product $(\,,)$.
Since $T_g\Mc(X)=S^2(X)$,  (\ref{be}) says that the orthogonal  complement of $T_gO(g)$ in $T_g\Mc(X)$ is $\ker \delta_g$.

\begin{remark}\label{weitexot}
{\rm The operators  $\delta_g$, $\delta_g^*$ and $\Rcc_g$ appear in
a Weitzenb\"ock type decomposition associated to the operator
$S_r:\Ac^r(X)\to \Ac^{r+1}(X)$ defined by
\[
(S_r\eta)(x_1,\cdots, x_{r+1})=\sum_i(\nabla_{x_i}\eta)(x_1,\cdots, \hat x_i,\cdots,x_{r+1})
\]
and its adjoint
\[
(S_r^*\eta)(x_1,\cdots,x_r)=-\sum_i(\nabla_{e_i}\eta)(e_i,x_1,\cdots,x_r).
\]
A straightforward computation gives
\[
(S_2^*S_2-S_1S_1^*)h=\nabla^*\nabla h+2\!\Rcc_g\!\! h-2h\circ {\rm Ric}_g, \quad h\in S^2(X).
\]
In particular, if $\delta_gh=0$ then
\begin{equation}\label{weitzen}
S_2^*S_2h=\nabla^*\nabla h+2\!\Rcc_g\!\! h-2h\circ {\rm Ric}_g,
\end{equation}
since $S_1^*=\delta_g$.
This formula plays a crucial role in our discussion of the rigidity of
$2k$-Einstein structures in Section \ref{demresrig}.
}
\end{remark}

\begin{definition}\label{geom}
A function $\Fc:\Mc(X)\to\mathbb R$ is a  {\em geometric functional} if  $\Fc(\phi^*g)=\Fc(g)$, for $g\in\Mc(X)$ and $\phi\in D(X)$.
\end{definition}

Thus, $\Fc$ is geometric if and only if it is constant along the orbits
of the $D(X)$-action on $\Mc(X)$. As important examples we single out the so-called
{\em Hilbert-Einstein-Lovelock functionals}:
\begin{equation}\label{hel}
\Fc^{(2k)}(g)=\int_X\Sc^{(2k)}_g\nu_g.
\end{equation}

By using Sobolev norms, we can make sense of when a geometric functional $\Fc$
is differentiable. In this case, for each $g$ there exists $a_g\in S^2(X)$
such that
\[
\dot\Fc_gh=(a_g,h), \qquad h\in S^2(X).
\]
We set $a_g={\rm grad}\,\Fc_g$, the {\em gradient} of  $\Fc$ at $g$.

It turns out that the Lovelock tensors in (\ref{locve}) are the gradients of the Hilbert-Einstein-Lovelock functionals, a result due to Lovelock \cite{Lo}.

\begin{proposition}\label{grad}\cite{Lo}\cite{L2}
In the notation above,
\begin{equation}\label{lovegrad}
{\rm grad}\,\Fc^{(2k)}_g=-\Jc^{(2k)}_g.
\end{equation}
\end{proposition}

\begin{proof}
It is shown in \cite{L2} that
\[
\dot\Sc^{(2k)}_gh=-\frac{1}{2(2k-1)!}\langle\Rc^{(2k)}_g,h\rangle+{\rm div}_g {\rm \omega},
\]
for some $\omega\in \Ac^1(X)$.
On the other hand, the classical Liouville formula says that
\begin{equation}\label{liouville}
\dot\nu_gh=\frac{1}{2}{\rm tr}_gh\nu_g,
\end{equation}
and the  result follows.
\end{proof}

The following proposition generalizes
(\ref{divinull}) and illustrates the importance of the decomposition
(\ref{be}) in the theory of geometric functionals.

\begin{proposition}\label{bianchiger}
If $\Fc$ is a differentiable geometric functional then its gradient is divergence free:
\begin{equation}\label{bianchiger2}
\delta_g{\rm grad}\,\Fc_g=0, \quad g\in\Mc(X).
\end{equation}
In particular,
\begin{equation}\label{bianchigerlove}
\delta_g\Jc_g^{(2k)}=0.
\end{equation}
\end{proposition}

\begin{proof}
Obvious in view of (\ref{be}).
\end{proof}

In the remainder of this section, we will use Proposition \ref{grad} to verify that the conditions of being $2k$-Einstein or having   $2k$-Gauss-Bonnet constant curvature both admit a variational interpretation; see \cite{L2}, \cite{L3} and \cite{Lo}.
For this purpose we define the normalized Hilbert-Einstein-Lovelock functionals  $\tilde\Fc^{(2k)}:\Mc(X)\to\mathbb R$,
\[
\tilde\Fc^{(2k)}(g)=\frac{\Fc^{(2k)}(g)}{\left(\int_X\nu_g\right)^
{\frac{n-2k}{2}}}.
\]
Note that $\tilde\Fc^{(2k)}$ is invariant under scalings.
Moreover, given a volume element $\mu$ in $X$, set
\[
\mathcal N_\mu=\{g\in\Mc_1(X);\mu_g=\mu\}.
\]

\begin{proposition}\label{equivalenteinstein}
The following statements with respect to a metric $g\in\Mc_1(X)$ are equivalent:
\begin{enumerate}
 \item $(X,g)$ is $2k$-Einstein;
 \item $g$ is a critical point of $\tilde\Fc^{(2k)}$;
 \item $g$ is a critical point of $\Fc^{(2k)}$ restricted to $\Mc_1(X)$;
 \item $g$ is a critical point of $\Fc^{(2k)}$ restricted to $\mathcal N_{\mu_g}$.
\end{enumerate}
\end{proposition}

\begin{proof} The equivalence between the second and third item is obvious. On the other hand, note that
\[
T_g\Mc_1(X)=\{h\in S^2(X);(g,h)=0\}
\]
and
\[
T_g\mathcal N_{\mu_g}=\{h\in S^2(X);{\rm tr}_gh=0\}.
\]
Thus, $g$ is a critical point of $\Fc^{(2k)}|_{\Mc_1(X)}$ (respectively, $\Fc^{(2k)}|_{{\mathcal N}_{\mu_g}}$) if and only if the orthogonal projection of ${\rm grad}\,\Fc^{(2k)}=-\Jc^{(2k)}$ onto $T_g\Mc_1(X)$ (respectively, $T_g\mathcal N_{\mu_g}$) vanishes. In both cases, there exists a function $\lambda$ in $X$ such that $\Rc^{(2k)}_g=\lambda g$. The result is now a consequence of Proposition \ref{constante}.
\end{proof}

If $g\in\Mc(X)$, we denote by $[g]=\{fg;f\in \Dc(X), f>0\}$ the {\em class of conformal metrics} to $g$. Moreover, if $g\in\Mc_1(X)$, we set
\[
[g]_1=\left\{\tilde g\in[g];\int_X\nu_{\tilde g}=1\right\}.
\]

\begin{proposition}\label{characgbconst}
A metric $g\in\Mc_1(X)$ has constant $2k$-Gauss-Bonnet curvature if and only if $g$ is a critical point for $\Fc^{(2k)}$ restricted to $[g]_1$.
\end{proposition}

\begin{proof} Observe first that, at least formally,
\begin{equation}\label{tangent}
T_g[g]=\left\{fg;f\in \Dc(X)\right\}
\end{equation}
and
\begin{equation}\label{tangent1}
T_g[g]_1=\left\{fg\in T_g[g];\int_Xf\nu_g=0\right\},
\end{equation}
so that the criticality of $g$ means that $(\Jc_g^{(2k)},fg)=0$ for all such $f$. Equivalently,
\[
(\Rc^{(2k)}_g,fg)=(2k-1)!(\Sc^{(2k)}_gg,fg).
\]
Recalling that $\langle h,g\rangle={\rm tr}_gh$ and using (\ref{cont}) we see that the criticality condition is given by
\[
2k\int_X\Sc^{(2k)}_gf\nu_g=n\int_X\Sc^{(2k)}_gf\nu_g,
\]
and since $2k<n$,
\[
\int_X\Sc^{(2k)}_gf\nu_g=0.
\]
Applying this to
\[
f=\Sc^{(2k)}_g-\int_X\Sc^{(2k)}\nu_g,
\]
it follows that
\[
\int_X(\Sc^{(2k)}_g)^2\nu_g=\left(\int_X\Sc^{(2k)}_g\nu_g\right)^2,
\]
that is, $\Sc^{(2k)}$ is constant.
\end{proof}

\section{Deformation and rigidity of $2k$-Einstein manifolds}
\label{def2keinstein}

In this section we will present some rigidity results for a class of $2k$-Einstein structures.
Let $X$ be a smooth, closed manifold of dimension $n\geq 5$.
The following definition captures the concept of a $2k$-Einstein structure.

\begin{definition}\label{modeins2k}
The {\em moduli space} of $2k$-Einstein structures in $X$ is the quotient
space
\[
{\mathfrak E}^{(2k)}(X)=\frac{E^{(2k)}(X)}{\mathbb R^+\times D(X)}=\frac{E^{(2k)}_1(X)}{D(X)}.
\]
Here, $E^{(2k)}(X)\subset \Mc(X)$ is the set of $2k$-Einstein metrics in $X$ and $E^{(2k)}_1=E^{(2k)}\cap \Mc_1(X)$.
In both cases the quotient map will be denoted by $g\mapsto \langle g\rangle $ and each class $\langle g\rangle$ is a $2k$-{\em Einstein} {\em structure} in $X$.
\end{definition}

Thus, a fundamental problem in this context is to determine the structure of $\mathfrak E^{(2k)}(X)$ for a given manifold  $X$. As in the case $k=1$, the first step would be to describe the space of {\em genuine} infinitesimal deformations of $2k$-Einstein structures. More precisely, if $\langle g\rangle\in\mathfrak E^{(2k)}(X)$ let $\langle g_t\rangle$, $t\in (-\epsilon,\epsilon)$, a differentiable one-parameter family of $2k$-Einstein structures with $g_0=g\in E^{(2k)}(X)$. As usual, we will think of  this family as a {\em deformation} of $\langle g\rangle$. In this case, and similarly to what happens in the Einstein case, the fact that each $g_t$ satisfies $\Rc^{(2k)}_{g_t}=\lambda_tg_t$ implies that
\[
h=\frac{d}{dt}g_t|_{t=0}\in S^2(X)
\]
satisfies
\begin{equation}\label{infini}
\dot\Rc^{(2k)}_gh=\lambda h,
\end{equation}
where $\lambda=\lambda_0$. Moreover,
since genuine infinitesimal deformations should be transversal to the orbits of $D(X)$, by
(\ref{be}) we must require that
\begin{equation}\label{div}
\delta_gh=0.
\end{equation}
Also, since we can assume, without loss of generality, that $g_t\in\Mc_1(X)$, we have as a consequence of (\ref{liouville}) that
\begin{equation}\label{normal}
\int_X{\rm tr}_gh\nu_g=0.
\end{equation}

At this point we are tempted to define the space of infinitesimal deformations of $\langle g\rangle$ by means of (\ref{infini}), (\ref{div}) and(\ref{normal}). We will see, however, that the last condition can be replace by an {\em algebraic} condition on $h$. The key point is the following theorem of J. Moser.

\begin{theorem}\label{moser}\cite{M}
If $g_1,g_2\in\Mc_1(X)$ then there exists $\phi\in D(X)$ such that $g_1=\phi^*g_2$.
\end{theorem}

In particular, $D(X)$ acts transitively on the space of metrics with the {\em same} volume element.
For this reason, and taking Proposition \ref{equivalenteinstein}, item 4, into account, in order to understand the structure of $\mathfrak E^{(2k)}(X)$ in a neighborhood $\langle g\rangle$, it suffices to consider the space of metrics
\[
\mathcal N_{g}=\{g'\in\Mc_1(X);\nu_{g'}=\nu_g\}
\]
with the same volume element as $g$, so that we will continue denoting by $\langle g'\rangle$ the corresponding $2k$-Einstein structure. But notice that, due to (\ref{liouville}), (\ref{normal}) now is replaced by
\begin{equation}\label{normal2}
{\rm tr}_gh=0.
\end{equation}

This discussion motivates the following definition.

\begin{definition}\label{defgenui2k}
If $(X,g)$ is $2k$-Einstein, $\Rc^{(2k)}_g=\lambda g$, the {\em space of infinitesimal deformations} of $\langle g\rangle$, denoted by  $\varepsilon_{\langle g\rangle}^{(2k)}$,  is the vector space of all elements $h\in\Cc^1(X)=S^2(X)$ such that
\begin{equation}\label{moduli}
\dot\Rc_g^{(2k)}h=\lambda h,
\end{equation}
and
\begin{equation}\label{moduli2}
\quad \delta_gh=0,\quad {\rm tr}_gh=0.
\end{equation}
\end{definition}

\begin{remark}\label{ellipt}
{\rm If we define $\mathcal I_g=\delta_g^{-1}(0)\cap{\rm {tr}}^{-1}_{g}(0)$ and
\begin{equation}\label{linearlg}
C_g^{(2k)}=\dot\Rc_g^{(2k)}-\lambda
\end{equation}
then $\varepsilon_{\langle g\rangle}^{(2k)}=\ker C_g^{(2k)}|_{\mathcal I_g}$. In particular, since $X$ is closed, $\varepsilon_{\langle g\rangle}^{(2k)}$ has finite dimension if $C_g^{(2k)}|_{\mathcal I_g}$
is an elliptic operator.}
\end{remark}

It follows from (\ref{ricciline})  and (\ref{lich}) that $C_g^{(2)}|_{\mathcal I_g}$ is always elliptic, that is, $\varepsilon_{\langle g\rangle}^{(2)}$ has finite dimension for {\em any} $(X,g)$ Einstein, a result due to Berger and Ebin \cite{BE}. However, if $k\geq 2$ the corresponding result is not necessarily true in general, for
$\varepsilon_{\langle g\rangle}^{(2k)}$ may be infinite dimensional for certain choices of $(X,g)$, which reflects the fact that $C_g^{(2k)}|_{\mathcal I_g}$ might be  of mixed type (not necessarily elliptic). In effect, consider the Riemannian product $X=M^r\times T^m$, where $M$ is an arbitrary Riemannian manifold and
$T^m$ is a flat torus. If $2k>r$ then $X$ is $2k$-Einstein independently of the metric in $M$, which shows that $\dim\varepsilon_{\langle g\rangle}^{(2k)}=+\infty$ in this case. This of course reflects the fact, already mentioned in the Introduction, that the symbol of $C^{(2k)}_g$ in general depends on the curvature tensor $R_g$.
In view of this, it is natural to look for examples of $2k$-Einstein structures $(X,g)$ for which $\dim\varepsilon_{\langle g\rangle}^{(2k)}<+\infty$.
Theorem \ref{main} below presents an interesting class of $2k$-Einstein structures for which this happens. First we need a definition.

\begin{definition}\label{classe}
Given integers $n$ and $k$ with $n\ge 5$ and $2\leq 2k< n$, and $\mu_k\neq 0$ a real number, we will denote by $\Hc_{n,k}$ the class of closed Riemannian manifolds $(X^n,g)$ of dimension $n$ which are $2k$-Einstein and have {\em constant} $(2k-2)$-sectional curvature, i.e satisfy
\begin{equation}\label{work}
R^{k-1}_g=\mu_k g^{2k-2};
\end{equation}
see Proposition \ref{lipkill}.
\end{definition}

Note that, as a consequence of Remark \ref{const}, the class $\Hc_{n,k}$ contains all non-flat space forms. Our first result shows that for $2k$-Einstein structures associated to elements of $\Hc_{n,k}$ the degeneracy phenomenon observed above does not happen.

\begin{theorem}\label{main}
If $(X,g)\in \Hc_{n,k}$ then $\varepsilon_{\langle g \rangle}^{(2k)}$ is finite dimensional.
\end{theorem}

Next we discuss the {\em rigidity} of $2k$-Einstein structures in the class  $\Hc_{n,k}$.

\begin{definition}\label{infnondef}
A $2k$-Einstein structure $\langle g\rangle\in\mathfrak E^{(2k)}(X)$ is said to be {\em infinitesimally} {\em non-deformable} if $\varepsilon_{\langle g\rangle}$ is trivial. Moreover, $\langle g \rangle$ is {\em non-deformable} if any deformation $\langle g_t\rangle$, $t\in (-\epsilon,\epsilon)$, is trivial, that is, $g_t=\phi_t^*g$, where $\phi_t\in D(X)$ with $\phi_0={\rm id}_X$. \end{definition}

Now define the constants
\begin{equation}\label{const1}
\underline \alpha_{n,k}=\frac{kn-5k+2}{n(kn+k+2-2n)}
\end{equation}
and
\begin{equation}\label{const2}
\overline \alpha_{n,k}=\frac{kn-2k-1}{n(kn-5k+n-1)},
\end{equation}
which are always {\em positive} if $k\geq 2$ and $n\geq 5$. The next result establishes a non-deformability criterium in terms of a certain assumption on the eigenvalues of $\Rcc_g|_{{\rm tr}_g^{-1}(0)}$. For this we define
\begin{equation}\label{pinching}
\underline a_0=\inf_{0\neq h\in {\rm tr}_g^{-1}(0)} \frac{({\Rcc}_gh,h)}{\| h\|^2},\quad \overline a_0=\sup_{0\neq h\in {\rm tr}_g^{-1}(0)} \frac{({\Rcc}_gh,h)}{\| h\|^2}.
\end{equation}

\begin{theorem}\label{main2}
If $(X,g)\in\Hc_{n,k}$
satisfies either $\underline a_0>\underline\alpha_{n,k}\kappa_g$ or $\overline a_0< \overline \alpha_{n,k}\kappa_g$, where $\kappa_g=2\Sc_g^{(2)}$ is the scalar curvature of $g$, then $\langle g\rangle $ is infinitesimally non-deformable.
\end{theorem}

\begin{corollary}\label{constcurv}
If $(X,g)$ is a space form with sectional curvature $\mu\neq 0$ then $\langle g\rangle $ is infinitesimally non-deformable.
\end{corollary}

\begin{proof}
It suffices to observe that, due to Remark \ref{casesp} below,  $\Rcc\!\!h=-\mu h$ if ${\rm tr}_gh=0$, so that $\underline a_0=\overline a_0=-\mu$. Since $\kappa_g=n(n-1)\mu$, the result follows readily.
\end{proof}

Adapting an argument in \cite{K1} one easily verifies that $\langle g\rangle$ infinitesimally non-deformable implies that $\langle g\rangle$ is non-deformable, which can  be applied, in particular, to the $2k$-Einstein structures in Theorem \ref{main2}. However, it is possible from the conclusion of this theorem to derive stronger  {\em rigidity} properties for the given structure. To explain this we recall that the decomposition (\ref{be}) implies the existence of a {local slice} $\mathcal V_g$ for the action of $D(X)$ in $\Mc(X)$ in a neighborhood $g$; see \cite{E}.

\begin{definition}\label{modulilocal}
The set of all $2k$-Einstein structures in $\mathcal V_g$ is called the {\em pre-moduli space} in a neighborhood of $\langle g\rangle$ and denoted by $\mathfrak E^{(2k)}_g(X)$.
\end{definition}

The moduli space itself, $\mathfrak E^{(2k)}(X)$, can be locally obtained from $\mathfrak E^{(2k)}_g(X)$ after passing to the quotient by the action of the isometry group of $(X,g)$, which is a  compact Lie group. However, we shall completely  ignore this issue and  deal directly with $\mathfrak E^{(2k)}_g(X)$. In particular, the definition below captures the notion of (local) rigidity of $2k$-Einstein structures.

\begin{definition}\label{rigidity}
$\langle g\rangle$ is {\em rigid} if it is an isolated element in $\mathfrak E^{(2k)}_g(X)$.
\end{definition}

The next result provides examples of rigid $2k$-Einstein structures.

\begin{theorem}\label{main3}
Under the conditions of Theorem \ref{main2}, $\langle g\rangle$ is rigid.
\end{theorem}

\begin{corollary}\label{constcurv2}
If $(X,g)$ is a space form of sectional curvature $\mu\neq 0$ then $\langle g\rangle $ is rigid.
\end{corollary}

Actually, Theorem \ref{main3} is a straightforward consequence of a more general result that elucidates the local structure of $\mathfrak E^{(2k)}_g(X)$, with $\langle g\rangle$ under the conditions of Theorem \ref{main}.

\begin{theorem}\label{main4}
 If $\langle g\rangle$ satisfies the assumptions of Theorem \ref{main} then $\mathfrak E^{(2k)}_g(X)$ has, in a neighborhood of $\langle g\rangle$, the structure of an analytical subset contained in a analytical manifold whose tangent space in $\langle g\rangle$ is precisely $\varepsilon^{(2k)}_{\langle g\rangle}$.
\end{theorem}

\begin{remark}\label{general}
{\rm For the spherical case ($\mu>0$), Corollaries \ref{constcurv}} and \ref{constcurv2} were first obtained in \cite{dLS1}.
\end{remark}

\section{Linearizing the $2k$-Ricci tensor}
\label{lineariricc}

The proofs of the finiteness and rigidity results stated in the previous section rely on a calculation of the operator $C^{(2k)}_g$ defined in Remark \ref{ellipt} above, which by its turn rests on the linearization of the $2k$-Ricci map $g\in\Mc(X)\mapsto {\Rc}^{(2k)}_g\in \Cc^1(X)$; see Proposition \ref{linearizado} below. We start by recalling some preliminary results proved in \cite{L1}, \cite{L2} and \cite{dLS1}.

For $h\in \Cc^1(X)$ it is defined in \cite{L2} the linear mapping
$F_h:\Cc^r(X)\to \Cc^r(X)$
as follows: for any $p\in X$ and $\{e_1,\ldots ,e_n\}$ an orthonormal basis of $T_pX$ diagonalizing $h$, set
\begin{eqnarray*}
&&(F_h\omega)(e_{i_1}\wedge \ldots \wedge e_{i_r},e_{j_1}\wedge \ldots
\wedge
e_{j_r})=\\
&&\qquad
=\Big(\sum_{k=1}^rh(e_{i_k},e_{i_k})+\sum_{k=1}^rh(e_{j_k},e_{j_k})
\Big)\omega(e_{i_1}\wedge \ldots \wedge e_{i_p},e_{j_1}\wedge \ldots
\wedge e_{j_r}).
\end{eqnarray*}
Consider also the operator
$\Dc^2:\Cc^1(X)\to \Cc^2(X)$
given by
\begin{eqnarray}
 \Dc^2 h (x_1\wedge x_2,y_1\wedge y_2)&=& \nabla^2_{y_1,x_1}h(x_2,
y_2)+ \nabla^2_{x_1,y_1}h(x_2, y_2)\nonumber\\
&&+\nabla^2_{y_2,x_2}h(x_1,y_1)+\nabla^2_{x_2,y_2}h (x_1, y_1)\nonumber\\
&&-\nabla^2_{y_1,x_2}h (x_1,y_2)-\nabla^2_{x_2,y_1}h (x_1, y_2)\nonumber\\
&&-\nabla^2_{y_2,x_1}h(x_2,y_1)-\nabla^2_{x_1,y_2}h (x_2, y_1) \label{d2},
\end{eqnarray}
where
\[
\nabla^2_{x,y}h=\nabla_x\nabla_yh-\nabla_{\nabla_xy}h
\]
is the usual Hessian operator.
The relevance of these concepts is illustrated by the following lemma, proved in \cite{L2}.

\begin{lemma}\label{prel}
The linearization of the curvature tensor is given by
\begin{equation}
\dot R_g h=-\frac14 \Dc^2 h+\frac14 F_h(R_g). \label{curvlin}
\end{equation}
\end{lemma}

We show in Corollary \ref{maincorollary} below that if $(X,g)\in \Hc_{n,k}$ then the second order term in
$
C^{(2k)}_g=\dot{\Rc}^{(2k)}_g-\lambda
$
is completely determined by the first and second order contractions of $\dot R_g$. Hence, in view of (\ref{curvlin}), it is crucial to determine such contractions for $\Dc^2h$ and $F_h(R_g)$. Such  a calculation has been carried out in
\cite{dLS1}.

\begin{proposition}\label{mainlema}\cite{dLS1}
For any metric $g\in\Mc(X)$, and given $h\in \Cc^1(X)$, the following identities hold:
\begin{enumerate}
\item If $\nabla^*\nabla$ is the Bochner Laplacian acting on $\Cc^1(X)$ then
\begin{equation}
c_g\Dc^2 h=-2\nabla^*\nabla h+2\nabla^2{\rm tr}_gh+4\delta^{*}_g\delta_g h-(\Rc^{(2)}_g\circ h+h\circ \Rc^{(2)}_g) +2\!\Rcc_g\!\! h;
\end{equation}
\item If $\Delta_g$ is the Laplacian associated to  $g$ then
\begin{equation}
c_g^2\Dc^2 h = -4\Delta_g {\rm tr}_gh-4\delta_g \delta_g h;
\end{equation}
\item $c_gF_h(R_g)=\Rc^{(2)}_g\circ h+h\circ \Rc_g^{(2)}+ 2\!\Rcc_g\!\!h$;
\item $c_g^2F_h(R_g)=4\langle \Rc_g^{(2)},h\rangle$.
\end{enumerate}
Here, $\Rc^{(2)}_g={\rm Ric}_g$ is the Ricci tensor.
\end{proposition}

\begin{corollary}\label{maincorollary}
The first and second order contractions of $\dot R_gh$ are respectively given by
\begin{equation}
c_g\dot R_gh=\frac12 \left(\nabla^*\nabla h-\nabla^2{\rm tr}_gh-2\delta^{*}_g\delta_g h+(\Rc_g^{(2)}\circ h+h\circ \Rc_g^{(2)}) \right) \label{qq}
\end{equation}
and
\begin{equation}
c_g^2 \dot R_gh=\Delta_g {\rm tr}_gh+\delta_g \delta_g h+ \langle\Rc_g^{(2)},h\rangle.\label{rr}
\end{equation}
\end{corollary}

\begin{remark}\label{caseins}
{\rm If $g$ is Einstein, ${\Rc}_g^{(2)}=\dfrac{\kappa_g}n g$, then
\[
{\Rc}_g^{(2)}\circ h+h\circ {\Rc}_g^{(2)}=\frac{2\kappa_g}n
h,
\]
from which we see that
\[
c_g(\Dc^2 h)=-2\nabla^*\nabla h+2\nabla^2 {\rm tr}_gh+4\delta^{*}_g\delta_g h-\frac{2\kappa_g}n
h +2\!\Rcc_g\!\! h.
\]}
\end{remark}


With these preliminaries at hand, we now turn to the context of Theorem \ref{main}, so that $(X,g)$ is a $2k$-Einstein manifold with
\begin{equation}\label{curvcond}
R_g^{k-1}=\mu_k g^{2k-2}, \quad \mu_k\neq 0.
\end{equation}
Observe that  if $k=1$ the $2k$-Einstein condition reduces to the usual Einstein case, $\R_g^{(2)}=\lambda g$, while if $k=2$, (\ref{curvcond}) means that $(X,g)$ is a space-form; see Remark \ref{const}.
This case has already been treated in \cite{dLS1} so we may assume  from now on that
$k>2$. Note that for metrics satisfying (\ref{curvcond}), (\ref{commutation}) implies
\begin{eqnarray*}
\R^{(2k)}_g
& = & \mu_k c_g^{2k-1}(g^{2k-2}R_g)\\
&=&(2k-2)!\mu_k\sum_{r=1}^{2k-2 }C^{2k-1}_r
\prod_{i=0}^{r-1}(n-3-i)\frac{g^{2k-2-r}}{(2k-2-r)!}c_g^{2k-1-r}R_g\\
&=&(2k-2)!\mu_k C^{2k-1}_{2k-3} \prod_{i=0}^{2k-4}(n-3-i)gc_g^2R_g
+\\
& & \quad +(2k-2)!\mu_k C^{2k-1}_{2k-2}
\prod_{i=0}^{2k-3}(n-3-i)c_gR_g\\
&=&(2k-1)!\prod_{i=0}^{2k-4}(n-3-i)\mu_k \Big((k-1)\kappa_g g +(n-2k)\R_g\Big),
\end{eqnarray*}
that is,
\begin{equation} \label{2kricci}
\R^{(2k)}_g=\frac{(2k-1)!(n-3)!}{(n-2k)!}\mu_k \left((k-1)\kappa_g g
+(n-2k)\R_g^{(2)}\right).
\end{equation}

\begin{proposition}\label{equivform}
If $(X,g)\in \Hc_{n,k}$  then:
\begin{enumerate}
 \item $(X,g)$ is $2k$-Einstein if and only if it is Einstein;
 \item $(X,g)$ has constant $2k$-Gauss-Bonnet curvature if and only if it has constant scalar curvature.
\end{enumerate}
In particular, if $(X,g)$ is $2k$-Einstein, that is, $\R^{(2k)}_g=\lambda g$,
then
\begin{equation}\label{boler}
\lambda=\frac{(2k)!}{n}\Sc^{(2k)}_g= \frac{(n-2)!(2k)!}{(n-2k)!2n}\mu_k
\kappa_g=\frac{k(n-2)}{n}C_{n,k}\mu \kappa_g,
\end{equation}
where
\begin{equation}
C_{n,k}=\frac{(2k-1)!(n-3)!}{(n-2k)!}.\label{constant}
\end{equation}
\end{proposition}

\begin{proof}
The first item follows from (\ref{2kricci}). On the other hand, contracting both sides of this expression and taking into account that $c_g g=n$ and $c_g\R_g^{(2)}=\kappa_g$,
\[
c_g\R^{(2k)}_g= (2k)!\Sc^{(2k)}_g=
\frac{(2k)!(n-2)!}{(n-2k)!2}\mu_k \kappa_g,
\]
from which
\begin{equation}\label{2k-scalar curv}
\Sc^{(2k)}_g=\frac{(n-2)!}{(n-2k)!2}\mu_k \kappa_g,
\end{equation}
and the second item follows straightforwardly.
\end{proof}

We are finally in conditions to linearize the $2k$-Ricci tensor under the assumption $(X,g)\in\Hc_{n,k}$.
In view of (\ref{def}) we get, for an arbitrary metric $g$,
\begin{equation}\label{sy}
\dot\Rc^{(2k)}_gh = (2k-1)(\dot c_gh)c^{2k-2}_gR^k_g +
kc^{2k-1}_gR^{k-1}_g\dot R_gh,
\end{equation}
which under (\ref{curvcond}) reduces to
\begin{eqnarray}
\dot\Rc^{(2k)}_gh &=& (2k-1)(\dot c_gh)c^{2k-2}_g(\mu g^{2k-2}R_g) +
kc^{2k-1}_g\mu g^{2k-2}\dot R_gh\nonumber \\
&=&(2k-1)\mu (\dot c_gh)c^{2k-2}_g(g^{2k-2}R_g) +
k\mu c^{2k-1}_g g^{2k-2}\dot R_gh.\label{Agh&Bgh}
\end{eqnarray}
We now identify the terms
\begin{equation}
A_gh=(2k-1)\mu (\dot c_gh)c^{2k-2}_g(g^{2k-2}R_g),\quad B_gh=k\mu c^{2k-1}_g g^{2k-2}\dot R_gh \label{var1},
\end{equation}
in the above expression.

Let us start with the first one. Using
(\ref{commutation}),
\begin{eqnarray*}
c_g^{2k-2}(g^{2k-2}R_g)& = & g^{2k-2}c_g^{2k-2}R_g
+(2k-2)!\times \\
&&\quad \times \sum_{r=1}^{2k-2}C^{2k-2}_r
\prod_{i=0}^{r-1}(n-4-i)\frac{g^{2k-2-r}}{(2k-2-r)!}c_g^{2k-2-r}R_g,
\end{eqnarray*}
and since $k>2$ it follows that
\begin{eqnarray*}
c^{2k-2}_gg^{2k-2}R_g & = & (2k-2)!\sum_{r=1}^{2k-2} C^{2k-2}_r\prod_{i=0}^{r-1}(n-4-i)\frac{g^{2k-2-r}}{(2k-2-r)!}c^{2k-2-r}_gR_g\\
&=&(2k-2)! C^{2k-2}_{2k-4}\prod_{i=0}^{2k-5}(n-4-i)\frac{g^2}{2}c^2_gR_g+\\
&& \quad +(2k-2)! C^{2k-2}_{2k-3}\prod_{i=0}^{2k-4}(n-4-i)gc_gR_g+\\
&&\qquad +(2k-2)! C^{2k-2}_{2k-2}\prod_{i=0}^{2k-3}(n-4-i)R_g,
\end{eqnarray*}
hence
\begin{eqnarray*}
c^{2k-2}_gg^{2k-2}R_g
&=&\frac{(2k-2)!(n-4)!}{(n-2k)!} \Big[(k-1)(2k-3)\frac{g^2}{2}c^2_gR_g+\\
&& \quad + (2k-2)(n-2k)gc_gR_g+(n-2k)(n-2k-1)R_g\Big]\\
&=&\frac{(2k-2)!(n-4)!}{(n-2k)!} \Big[(k-1)(2k-3)\Sc_g^{(2)}\frac{g^2}{2}+ \\
& & \quad +(2k-2)(n-2k)g\R_g^{(2)}+(n-2k)(n-2k-1)R_g\Big].
\end{eqnarray*}
Now observe that, by Proposition \ref{equivform}, item 1,
$g$ is Einstein, $\R_g^{(2)}=\dfrac{\kappa_g}n g$, so that the first two terms in the bracket contribute to
\[
\dfrac{(k-1)\kappa_g}{n}\Big((2k-3)n
 +  4(n-2k)\Big)\dfrac{g^2}2=\dfrac{(k-1)(2kn+n-8k)\kappa_g}{n}\dfrac{g^2}2,
\]
which leads to
\begin{eqnarray*}
c^{2k-2}_gg^{2k-2}R_g & = &
\frac{(2k-2)!(n-4)!}{(n-2k)!} \Big[\dfrac{(k-1)(2kn+n-8k)\kappa_g}{2n}g^2+\\
&& \quad +(n-2k)(n-2k-1)R_g\Big].
\end{eqnarray*}
Then applying  $\dot c_gh$ to this identity and comparing with (\ref{var1}) and (\ref{constant}), we obtain
\begin{equation}\label{prelim}
A_gh= \frac{C_{n,k}\mu}{(n-3)}\Big[\dfrac{(k-1)(2kn+n-8k)\kappa_g}{n}(\dot c_gh)\dfrac{g^2}2+(n-2k)(n-2k-1)(\dot c_gh)R_g\Big],
\end{equation}
that is, $A_gh$ has been determined up to the terms $(\dot c_gh)g^2/2$ and $(\dot c_gh)R_g$, which we now analyze.

Initially, linearizing the identity $c_g(g^2/2)=(n-1)g$ gives, after using (\ref{kulkarni}),
\begin{eqnarray*}
(n-1)h & = & (\dot c_gh)(g^2/2)+c_ggh\\
& = & (\dot c_gh)(g^2/2)+gc_gh+(n-2)h,
\end{eqnarray*}
whence
\begin{equation}\label{cgh}
(\dot c_gh)(g^2/2)=h-\tr_ghg.
\end{equation}
On the other hand,
since $c_gR_g=\Rc_g^{(2)}$, it follows after linearization that
\begin{equation}\label{1}
(\dot c_gh)R_g=\dot \R_g^{(2)} h-c_g\dot R_g h,
\end{equation}
and if we substitute (\ref{ricciline}) and (\ref{curvlin}) into the right-hand side, a cancelation yields
\begin{equation}\label{dotcgh}
(\dot c_gh)R_g=
-\Rcc_g\!\! h,
\end{equation}
so that if we take (\ref{cgh}) and (\ref{dotcgh}) to (\ref{prelim}) we get
\begin{equation}
A_gh= \frac{C_{n,k}\mu}{(n-3)}\Big[\dfrac{(k-1)(2kn+n-8k)\kappa_g}{n}(h-(\tr_gh)g)-(n-2k)(n-2k-1)
\!\Rcc_g\!\! h\Big] \label{Agh}.
\end{equation}

To determine $B_gh$ we will use (\ref{commutation}) with
$k>1$, an assumption implying  in particular that $c^{2k-1}_g\dot R_gh=0$, because $\dot R_gh\in \Cc^2(X$). Thus,
\begin{eqnarray*}
c^{2k-1}_g\left(\frac{g^{2k-2}}{(2k-2)!}\dot R_gh\right) & =  & \sum_{r=1}^{2k-2}C^{2k-1}_r
\prod_{i=0}^{r-1}(n-3-i)\frac{g^{2k-2-r}}{(2k-2-r)!}c^{2k-1-r}_g\dot R_gh\\
 &= &  (2k-1)(k-1)\prod_{i=0}^{2k-4}(n-3-i)gc^2_g\dot R_gh  +\\
 & & \quad +(2k-1) \prod_{i=0}^{2k-3}(n-3-i)c_g\dot R_gh,
\end{eqnarray*}
so that
\begin{eqnarray*}
c^{2k-1}_gg^{2k-2}\dot
R_gh=\frac{(2k-1)!(n-3)!}{(n-2k)!}\left((k-1)(c^2_g\dot R_gh)g+
(n-2k)c_g\dot R_gh\right).\label{aa}
\end{eqnarray*}
Thus, taking into account Corollary \ref{maincorollary}, we obtain
\begin{eqnarray*}
c^{2k-1}_gg^{2k-2}\dot
R_gh & = & C_{n,k}\Big\{(k-1)\Big[\Delta_g \tr_gh+\delta_g \delta_g h+ \langle\R_g^{(2)},h\rangle\Big]g + \\
& & \quad +
(n-2k)\frac12 \Big[\nabla^*\nabla h-\nabla^2\tr_gh-\\
& & \qquad -2\delta^{*}_g\delta_g h+\R_g^{(2)}\circ h+h\circ \R_g^{(2)} \Big]\Big\},
\end{eqnarray*}
which gives, after using the Einstein condition (see Proposition \ref{equivform}),
\begin{eqnarray*}
B_gh&=&k\mu_k C_{n,k}\Big\{(k-1)\Big[\Delta_g \tr_gh+\delta_g \delta_g h+ \frac{\kappa_g}{n}\tr_gh\Big]g  \\
&& \quad + (n-2k)\frac12 \Big(\nabla^*\nabla h-\nabla^2\tr_gh-2\delta^{*}_g\delta_g h+\frac{2\kappa_g}{n}h \Big)\Big\}.\label{Bgh}
\end{eqnarray*}
Hence, if we substitute (\ref{Agh}) and (\ref{Bgh}) into (\ref{Agh&Bgh}) we obtain, after some simplifications, the expression for the linearization of the  $2k$-Ricci tensor of $(X,g)\in \Hc_{n,k}$:

\begin{eqnarray}\label{ricci2kline}
\dot\Rc^{(2k)}_gh &=& \mu C_{n,k}\bigg\{ k(n-2k)\frac12 \Big(\nabla^*\nabla h-\nabla^gd\tr_g(h)-2\delta^{*}_g\delta_g h \Big)+\nonumber\\
&&\quad + k(k-1)\Big(\Delta_g \tr_g(h)+\delta_g \delta_g h\Big)g+\nonumber\\
&&\qquad +\frac{\kappa_g}{n}\Big\{-(k-1)\Big(k+\frac{n-2k}{n-3}\Big)(\tr_gh)g+\\
& &\quad \quad \quad  +\Big(k(n-2)+\frac{(k-1)(n-2k)}{n-3}\Big)h\Big\}-\nonumber\\
&&\quad \quad\quad \quad -\frac{(n-2k)(n-2k-1)}{(n-3)}\!\Rcc_g\!\! h\bigg\}.\nonumber
\end{eqnarray}
In view of this the next result follows readily.

\begin{proposition}\label{linearizado}
If $(X,g)\in \Hc_{n,h}$  then
\begin{eqnarray*}
C_g^{(2k)}h&=& \mu_k C_{n,k}\bigg\{ k(n-2k)\frac12 \Big(\nabla^*\nabla h-\nabla^2\tr_gh-2\delta^{*}_g\delta_g h \Big)+\\
&&\quad + k(k-1)(\Delta_g \tr_gh+\delta_g \delta_g h)g+\\
&&\qquad +\frac{(n-2k)}{(n-3)}\Big\{\alpha(n,k)\kappa_g \tr_ghg+\frac{(k-1)\kappa_g}n h-(n-2k-1)\!\Rcc_g\!\! h\bigg\}.
\end{eqnarray*}
where $\alpha(n,k)$ depends only on $n$ and $k$. In particular,
\begin{eqnarray*}
C_g^{(2k)}|_{\tr_g^{-1}(0)}h&=&
\mu_k C_{n,k}\bigg\{ k(n-2k)\frac12 \Big(\nabla^*\nabla h-2\delta^{*}_g\delta_g h \Big)+\\
&&\quad + k(k-1)(\delta_g \delta_g h)g+\\
&&\qquad +\frac{(n-2k)}{(n-3)}\Big(\frac{(k-1)\kappa_g}n h-(n-2k-1)\!\Rcc_g\!\! h\Big)\bigg\},
\end{eqnarray*}
and
\begin{equation}
C_g^{(2k)}|_{\mathcal I_g}h=\mu_k k(n-2k)C_{n,k}\bigg\{ \frac12 \nabla^*\nabla h +P_gh\bigg\}, \label{restrictoper}
\end{equation}
where
\begin{equation}\label{restrictpot}
P_gh=\frac{(k-1)\kappa_g}{nk(n-3)} h-\frac{(n-2k-1)}{k(n-3)}\Rcc_g\!\! h.
\end{equation}
\end{proposition}

\begin{remark}\label{casesp}
{\rm If $g$ has constant sectional curvature $\mu$, so that its curvature tensor is $R_g=\frac{\mu}{2}g^2$, then $\kappa_g=n(n-1)\mu$, $\mu_k=\mu^{k-1}/2^{k-1}$ and $\Rcc\!\! h =\mu((\tr_gh)g- h)$, thus implying
$P_gh=\mu h$. Consequently,
\begin{equation*}
L_g^{(2k)}|_{\mathcal I_g}h= \mu k(n-2k)C_{n,k}\left( \frac{1}{2} \nabla^*\nabla h +\mu h\right),
\end{equation*}
which retrieves a result previously obtained in \cite{dLS1}.}
\end{remark}

\section{Proving the rigidity theorems}\label{demresrig}

In this section we will make use of the linearization formulae in Proposition \ref{linearizado} in order to prove the rigidity theorems stated in Section \ref{def2keinstein}.

Observe that Theorem \ref{main} is a straightforward consequence of (\ref{restrictoper}), since $\nabla^*\nabla$ is elliptic; see Remark \ref{ellipt}. To prove Theorem \ref{main2}, assume that $(X,g)\in\Hc_{n,k}$  and let $h\in \varepsilon^{(2k)}_{\langle g\rangle}$, so that from (\ref{restrictoper}) $h$ satisfies
\begin{equation}\label{satis}
\nabla^*\nabla h+2P_gh=0.
\end{equation}
Also note that the Hodge Laplacian $\Delta^\nabla$ in $S^2(X)=\Ac^1(X;\Lambda^1(X))$ admits the Weitzenb\"ock decomposition
\begin{equation}\label{weit2}
\Delta^\nabla h=\nabla^*\nabla h-\Rcc_gh+h\circ {\rm Ric}_g;
\end{equation}
see \cite{Be}.
Thus, using (\ref{satis}), (\ref{weit2}), the fact that $g$ is Einstein (by Proposition \ref{equivform}) and (\ref{restrictpot}), we obtain
\begin{eqnarray*}
0 & \leq & \|d^\nabla h\|^2+\|\delta^\nabla h\|^2\\
  & = & \left(\left(\nabla^*\nabla-\!\Rcc_ g+\frac{\kappa_g}{n}\right)h,h\right)\\
  & = & \left(\left(-2P_g-\!\Rcc_g+\frac{\kappa_g}{n}\right)h,h\right)\\
  & = & \frac{1}{k(n-3)}\left(\left(\frac{kn-5k+2}{n}\kappa_g+
         \left(2n-k-2-kn\right)\!\Rcc_g)\right)h,h\right)\\
  &\leq &   \frac{1}{k(n-3)}\left(\left(\frac{kn-5k+2}{n}\kappa_g+
         \left(2n-k-2-kn\right)\underline a_0)\right)h,h\right),
\end{eqnarray*}
where we used that $2n-k-2-kn<0$ if $k\geq 2$. Hence,
if $h\neq 0$ then it necessarily holds that
\[
\frac{kn-5k+2}{n}\kappa_g+\left(2n-k-2-kn\right)\underline a_0\geq 0,
\]
that is,
\[
\underline a_0\leq \underline\alpha_{n,k}\kappa_g.
\]
On the other hand, a similar reasoning with (\ref{weit2}) replaced by (\ref{weitzen}) gives
\begin{eqnarray*}
0 & \leq & \|S_2 h\|^2\\
  & = & \left(\left(\nabla^*\nabla+2\!\Rcc_g -2\frac{\kappa_g}{n}\right)h,h\right)\\
  & = & \left(\left(-2P_g+2\!\Rcc_g-2\frac{\kappa_g}{n}\right)h,h\right)\\
  & = & \frac{2}{k(n-3)}\left(\left(-\frac{kn-2k-1}{n}\kappa_g+
         \left(kn-5k+n-1\right)\!\Rcc_g)\right)h,h\right)\\
  &\leq &   \frac{2}{k(n-3)}\left(\left(-\frac{kn-2k-1}{n}\kappa_g+
         \left(kn-5k+n-1\right)\overline a_0)\right)h,h\right),
\end{eqnarray*}
that is, $h\neq 0$ necessarily leads to
\[
\overline a_0\geq \overline \alpha_{n,k}\kappa_g.
\]
This completes the proof of Theorem \ref{main2}.

We now proceed to the proof of Theorem \ref{main4}, from which Theorem \ref{main3} follows, as already explained. Observe initially that the fact that the Lovelock tensor $\Jc_g^{(2k)}$ is divergence free, for {\em any} metric $g$, can be expressed as
\begin{equation}\label{free}
\delta_g\Rc_g^{(2k)}+(2k-1)!d\Sc_g^{(2k)}=0.
\end{equation}
If we introduce the functional $\mathcal G:\Mc_1(X)\to \Cc^1(X)$,
\[
\mathcal G(g)=\Rc_g^{(2k)}-\frac{(2k)!}{n}\Ac^{(2k)}(g)g,
\]
and the $(2k)$-{\em Bianchi} {\em operator} $\beta^{(2k)}_g:\Cc^1(X)\to \Dc^1(X)$,
\[
\beta^{(2k)}_g=\delta_g+\frac{1}{2k}d{\mathsf {tr}}_g,
\]
the next proposition follows immediately.

\begin{proposition}\label{dr}
The following properties hold:
\begin{enumerate}
 \item $g$ is $2k$-Einstein if and only if $\mathcal G(g)=0$;
 \item If $g$ is $2k$-Einstein then $\dot{\mathcal G}_g=C_g^{(2k)}$ in $\Cc^1(X)$.
 \item For every $g$, $\beta^{(2k)}_g\mathcal G_g=0$. In particular, if $g$ is $2k$-Einstein,
 $\beta^{(2k)}_g\dot{\mathcal G}_g=0$.
\end{enumerate}
\end{proposition}

This proposition gives, in particular, the identification $E^{(2k)}_1(X)={\mathcal G}^{-1}(0)$. In this context, the
identity $\beta^{(2k)}_g\dot{\mathcal G}_g=0$ means that $\dot{\mathcal G_g}$ is not surjective, since every $u\in\im\,  \dot{\mathcal G_g}$ belongs to the kernel of the first order operator $\beta^{(2k)}_g$, which of course reflects the diffeomorphism invariance of the $2k$-Einstein condition. Obviously, this is a serious complication when trying to use the Implicit Function Theorem to probe the local structure of $\mathfrak E^{(2k)}_g(X)$.
A way out  is to use  Proposition \ref{linearizado} and consider, for
$h\in {\rm tr}_g^{-1}(0)$,
the operator
\begin{eqnarray*}
\tilde C_g^{(2k)}h & = & C_g^{(2k)}h+\mu C_{n,k}k(n-2k)\delta^{*}_g\delta_g h - \mu C_{n,k}k(k-1)(\delta_g \delta_g h)g \\
&=&
\mu k(n-2k)C_{n,k}\bigg\{ \frac12 \nabla^*\nabla h +P_gh\bigg\}.
\end{eqnarray*}

\begin{lemma}\label{invar}
$\tilde C_g^{(2k)}$ leaves invariant the subspace
\[
T_g\Mc_1(X)=\left\{h\in\Cc^1(X);\in \int_X{\rm tr}_g h\nu_g=0\right\}.
\]
In particular, $\tilde C_g^{(2k)}(T_g\Mc_1(X))$ is closed.
\end{lemma}

\begin{proof}
It is easy to see that ${\rm tr}_g\!\Rcc_g\!\! h=\langle \Rc^{(2)}_g,h\rangle$, so that the Einstein condition implies
\[
{\rm tr}_g P_gh=
 \beta (n,k)\kappa_g{\rm tr}_g h, \quad \beta(n,k)=\frac{n-k}{nk(n-3)}.
\]
Since
${\rm tr}_g\nabla^*\nabla h=\Delta_g{\rm tr}_gh$, it follows that
\[
{\rm tr}_g\tilde C_g^{(2k)}h=(n-2k)kC_{n,k}\mu\left(\frac{1}{2}\Delta_g{\rm tr}_gh
+\beta (n,k)\kappa_g{\rm tr}_g h \right),
\]
from which
\[
\int_X{\rm tr}_g\tilde C_g^{(2k)}h\,\nu_g=(n-2k)kC_{n,k}\beta (n,k)\mu\kappa_g\int_X{\rm tr}_gh\,\nu_g,
\]
which proves the invariance of $T_g\Mc_1(X)$. The ellipticity of $\tilde C_g^{(2k)}$ then implies that $\tilde C_g^{(2k)}(T_g\Mc_1(X))$ is closed.
\end{proof}

We now verify the constraints imposed on $\tilde C _g^{(2k)}$ by the diffeomorphism invariance of the $2k$-Einstein condition. Using  Proposition  \ref{dr} and the identity ${\rm tr}_g\delta_g^*\eta=-\delta_g\eta$, $\eta\in\Ac^1(X)$, we get
\begin{eqnarray*}
\beta^{(2k)}_g\tilde C_g^{(2k)}h & = & \mu C_{n,k}k(n-2k)\beta^{(2k)}_g\delta^{*}_g\delta_g h - \mu C_{n,k}k(k-1)\beta^{(2k)}_g[(\delta_g (\delta_g h))g] \\
&=& C_{n,k}k(n-2k)\mu\Big\{\delta_g(\delta^{*}_g\delta_g h)+\frac{1}{2k}d{\rm tr}_g(\delta^{*}_g\delta_g h)\Big\}-\\
 &&\quad -  C_{n,k}k(k-1)\mu\Big\{\delta_g[(\delta_g (\delta_g h))g]+\frac{1}{2k}d{\rm tr}_g[(\delta_g (\delta_g h))g]\Big\} \\
&=&C_{n,k}k(n-2k)\mu\Big\{\delta_g(\delta^{*}_g\delta_g h)-\frac12d(\delta_g (\delta_g h))\Big\},
\end{eqnarray*}
so that setting $G_g=\delta_g\delta_g^*-\frac12d\delta_g$ it follows that
\begin{equation}\label{pup}
\beta^{(2k)}_g\tilde C_g^{(2k)}h=C_{n,k}k(n-2k)\mu G_g (\delta_g h),
\end{equation}
with $G_g$ being elliptic.

Now, (\ref{pup}) initially gives $\tilde{C}_g^{(2k)}(T_g\mathcal V_g)\subset {\ker}\beta^{(2k)}_g$.
Moreover, if $k=\tilde{C}_g^{(2k)} h\in \ker\beta^{(2k)}_g$, $h\in T_g{\mathcal M}_1(X)$, then $\delta_gh\in\ker G$, a space  of {\em finite} dimension, and this gives
\[
\tilde{C}_g^{(2k)}(T_g\mathcal{V}_g)\subset\tilde{C}_g^{(2k)}(T_g{\mathcal M}_1(X)\cap\ker \beta^{(2k)}_g)\subset
\tilde{C}_g^{(2k)}\left(T_g{\mathcal M}_1(X)\cap\delta_g^{-1}\ker G\right).
\]
As $T_g\mathcal{V}_g$ is closed and has {\em finite} dimension in $T_g{\mathcal M}_1(X)\cap\delta_g^{-1}\ker G$, it is easy to verify that
$\tilde{C}_g^{(2k)}(T_g\mathcal{V}_g)$
is closed in ${C}_g^{(2k)}\left(T_g{\mathcal M}_1(X)\cap\delta_g^{-1}\ker G\right)$. Thus,  $\tilde{C}_g^{(2k)}(T_g{\mathcal M}_1(X))$ is closed in $\tilde{C}_g^{(2k)}(T_g{\mathcal M}_1(X)\cap\ker \beta^{(2k)}_g)$, which is closed in ${\mathcal C}^1(X)$.
We conclude that, although $C_g^{(2k)}=\dot {\mathcal G}_g:T_g\mathcal{V}_g\to {\mathcal C}^1(X)$ is not surjective, its range is closed .
Hence, if $\pi$ is the orthogonal projection of ${\mathcal C}^1(X)$ onto ${ C}_g^{(2k)}(T_g\mathcal{V}_g)$,
the composition $\pi\circ {\mathcal G}:\mathcal{V}_g\to {C}_g(T_g\mathcal{V}_g)$, which is analytic, is a submersion in $g$. Thus, $(\pi\circ {\mathcal G})^{-1}(0)$ is a real analytical manifold in a neighborhood of $g$ having $\mathfrak{E}_{g}^{(2k)}$ as its tangent space in $g$. In this manifold, the mapping ${\mathcal G}$ is analytical so  that the pre-moduli space ${\mathfrak E}^{2k}_g(X)={\mathcal G}^{-1}(0)$ is an analytic subset. This completes the proof of Theorem \ref{main4} and, therefore, of Theorem \ref{main3}.

\section{The Yamabe problem for Gauss-Bonnet curvatures}
\label{yamabegb}

In this section we consider a generalization of the classical Yamabe problem, namely, the Yamabe problem for the Gauss-Bonnet curvature $\Sc^{(2k)}$. As explained below, in the class of locally conformally flat manifolds this problem is equivalent to the so-called $\sigma_k$-Yamabe problem and  has already been considered under a certain ellipticity assumption on the { background}  metric (see \cite{GW}, \cite{LL}). As a consequence of a formula for the linearization  of the Gauss-Bonnet curvature on $(X,g)\in\Hc_{n,k}$ (see Proposition \ref{linescar} below) we shall prove a local version of the  Yamabe problem for the Gauss-Bonnet curvatures in a neighborhood of a subclass of $\Hc_{n,k}$ which includes {\em all} nonflat space forms, except for the round sphere. This gives, in particular, many  examples of {\em background} metrics with {\em non-null} Weyl tensor for which this Yamabe type problem is affirmatively solved.


The classical Yamabe problem asks for the existence of a metric with constant scalar curvature  in each conformal class of metrics in a closed Riemannian manifold of dimension $n\geq 3$; see \cite{LP}.
As it is evident from Proposition \ref{characgbconst}, this problem is just the first in a series of variational problems in Conformal Geometry. More precisely, it is also natural to consider the following problem:

\vspace{0.4cm}
\noindent
{\bf Yamabe Problem for Gauss-Bonnet curvatures:} {\em given $n\geq 4$, $1\leq k\leq n/2$ and a Riemannian manifold $(X^n,g)$, does there exist $g'\in [g]$, such that $\Sc^{(2k)}_{g'}$ is constant?}

\vspace{0.4cm}

Clearly, for $k=1$ this reduces to the classical Yamabe problem. On the other hand, the above general problem is related to another problem of Yamabe type extensively studied recently. To see this, recall the following decomposition  for the curvature tensor:
\begin{equation}\label{schout2}
R_g=A_g\odot g+W_g,
\end{equation}
where $W_g$ is the Weyl tensor and
\begin{equation}\label{schouten}
A_g=\frac{1}{n-2}\left({\rm Ric}_g-\frac{\kappa_g}{2(n-1)}\right)g,
\end{equation}
is the {\em Schouten tensor}. Since $W_g$ is a conformal invariant, all the information regarding conformal changes of metrics is encoded in $A_g$. Thus, if $\sigma_k(A_g)$ denotes the $k^{\rm th}$ elementary symmetric function of the eigenvalues of $A_g$ (understood as an element of $\Tc^{(1,1)}(X)$) the following problem becomes rather natural:

\vspace{0.4cm}
\noindent
{\bf $\sigma_k$-Yamabe Problem:}
{\em given $n\geq 4$, $1\leq k\leq n/2$ and a Riemannian manifold $(X^n,g)$, does there exist $g'\in [g]$, such that $\sigma_k(A_{g'})$ is constant?}

\vspace{0.4cm}

Notice that since $\sigma_1(A_g)$ is a multiple of $\kappa_g$, this reduces to the classical Yamabe problem for $k=1$; see
\cite{V} for a nice introduction to the $\sigma_k$-Yamabe problem.

Actually, the Yamabe type problems above are completely equivalent in the class of conformally flat manifolds. This follows from the proposition below, proved in \cite{L3}.

\begin{proposition}\label{doisyam}
If $n\geq 4$, $1\leq k\leq n/2$ and $(X^n,g)$ is a Riemannian manifold, then
\begin{equation}\label{gbsigk}
\Sc^{(2k)}_g=\frac{(n-k)!k!}{(n-2k)!}\sigma_k(A_g)+\sum_{i=0}^{k-1}\frac{k!}
{i!(k-i)!(n-2k)!}\langle\star g^{n-2k+i}P_g^i,W^{k-i}_g\rangle,
\end{equation}
where $\star$ is the natural extension of the Hodge star operator acting on $\Ac^{\bullet,\bullet}(X)$. In particular, if $(X,g)$ is locally  conformally flat ($W_g=0$) then
\begin{equation}\label{gbsigk2}
\Sc^{(2k)}_g=\frac{(n-k)!k!}{(n-2k)!}\sigma_k(A_g).
\end{equation}
\end{proposition}

The $\sigma_k$-Yamabe problem for conformally flat manifolds (or, equivalently, the Yamabe problem for the Gauss-Bonnet curvatures) were considered in \cite{GW} and \cite{LL}, assuming that the {\em background} metric satisfies a certain ellipticity condition. The next theorem  solves the Yamabe problem for the Gauss-Bonnet curvatures in a neighborhood of Riemannian manifolds in the class $\Hc_{n,k}$, except for the round spheres, and provides many new examples of {\em non}-conformally flat manifolds for which this problem is affirmatively solved.

\begin{definition}\label{subclasse}
Given $n\geq 4$ and $1\leq k<n/2$, let $\Hc_{n,k}'$ be the complement of the set of round spheres in $\Hc_{n,k}$.
\end{definition}
Thus,
$(X,g)\in \Hc'_{n,k}$ if and only if $g$ is $2k$-Einstein and satisfies
\[
R_g^{k-1}=\mu_k g^{2k-2},\quad \mu_k\neq 0,
\]
with $(X,g)$
being isometrically distinct from a round sphere. Observe that in this case it follows from (\ref{2k-scalar curv}) that
\[
\Sc^{(2k)}_g=\frac{(n-2)!}{(n-2k)!2}\mu_k\kappa_g,
\]
and since $g$ is Einstein by Proposition \ref{equivform} we see that the $2k$-Gauss-Bonnet curvature of $(X,g)$ is constant.
Moreover, if $(X,\gm)\in\Hc'_{n,k}$, $\Dc^+(X)$ denotes the set of positive smooth functions in $X$ and $1:X\to\mathbb{R}$ is the function identically equal to
$1$.

\begin{theorem}\label{main5}
Assume that $4\leq 2k<n$ and let $(X,g_0)\in\Hc'_{n,k}$ with ${\rm vol}(X,g_0)=\nu$. Then the space $\Mc^{(2k)}_{\nu}(X)$ of the metrics in $X$ with  {\em constant} $2k$-Gauss-Bonnet curvature and volume $\nu$ has, in a neighborhood of $g_0$, the structure of an ILH-submanifold (of infinite dimension) of $\Mc(X)$. Moreover, the map  $\xi:\Dc^+(X)\times\Mc_{\nu}^{(2k)}(X)\to\Mc(X)$, given by $\xi(f,g)=fg$, is ILH-smooth in a neighborhood of $(1,g_0)$ and its derivative in $(1,g_0)$ is an isomorphism. In particular, there exists a neighborhood $U$ of $g_0$ in $\Mc(X)$ such that any metric in $U$ is conformal to some metric whose $2k$-Gauss-Bonnet curvature is constant.
\end{theorem}

For the ILH terminology we refer to \cite{O}. We also mention that the proof of this theorem is inspired on an argument due to N. Koiso \cite{K}, where a similar result has been proved in the case $k=1$ for a class of manifolds containing $\Hc'_{n,k}$.

\begin{remark}\label{fullyeq}{\rm
The local Yamabe type result in Theorem \ref{main5} does not hold true in case  $g_0$ is the round metric on the sphere. Indeed, if we pull-back $g_0$ using the flow of a conformal vector field we get a one-parameter family of metrics with the same Gauss-Bonnet curvature and volume.}
\end{remark}

\section{Linearizing the Gauss-Bonnet curvature}
\label{lin}

The main ingredient in the proof of Theorem \ref{main5}
is a formula for the linearization of the $2k$-Gauss-Bonnet curvature at Riemannian manifolds in the class $\Hc_{n,k}$; see  Proposition \ref{linescar} below. We start by observing that, by Definition \ref{rew},
\[
\Sc^{(2k)}_g=\frac{1}{(2k)!}c_g^{2k}R^{(k)}_g,
\]
from which we obtain
\begin{eqnarray*}
\dot\Sc^{(2k)}_gh & = & \frac{1}{(2k-1)!}\dot c_ghc_g^{2k-1}R^{k}_g+\frac{k}{(2k)!}c_g^{2k}R_g^{k-1}\dot R^{k}_gh\\
& = & \frac{1}{(2k-1)!}\dot c_gh\Rc^{(2k)}_g+\frac{k}{(2k)!}c_g^{2k}R_g^{k-1}\dot R^{k}_gh.
\end{eqnarray*}
Thus, if $(X,g)\in\Hc_{n,k}$ we can use (\ref{curvcond}), (\ref{2kricci}) and (\ref{constant}) to check  that
\begin{eqnarray}\label{svar}
\dot\Sc^{(2k)}_gh & = & \dfrac{C_{n,k}\mu_k}{(2k-1)!}\Big((k-1)\kappa_g (\dot c_gh) g
+(n-2k)(\dot c_gh)\R_g^{(2)}\Big)+ \nonumber\\
& & \quad+\dfrac{k}{(2k)!} \mu_k c_g^{2k} g^{2k-2}\dot R_gh.
\end{eqnarray}

Now notice first that
$c_gg=n$ implies, after linearization, that $(\dot c_gh)g+c_gh=0$, from which we get
\begin{equation}\label{A}
(\dot c_gh)g=-\tr_gh.
\end{equation}
On the other hand, the expressions $\Rc_g^{(2)}=c_gR_g$ and $\kappa_g=c_g\Rc^{(2)}_g$ lead to
\begin{eqnarray*}
(\dot c_gh)\Rc^{(2)}_g & = & \dot\kappa_gh-c_g\dot\Rc^{(2)}_gh\\
 & = & \dot\kappa_gh -c_g(\dot c_gh)R_g-c_g^2\dot R_gh.
\end{eqnarray*}
Also, $\kappa_g=c_g^2R_g$ implies
\[
\dot\kappa_gh=2c_g(\dot c_gh)R_g+c^2_g\dot R_gh,
\]
so that $(\dot c_gh)\Rc^{(2)}_g=c_g(\dot c_gh)R_g$ and  by (\ref{dotcgh}) and (\ref{jjj}),
\begin{equation}\label{B}
(\dot c_gh)\R_g=-\langle\R_g^{(2)},h\rangle.
\end{equation}
Finally, if we use (\ref{commutation}), taking into account that $c_g^{2k}\dot R_gh=0$ for $k>1$ and  $c_g^{2k-r}\dot R_gh=0$ for $r< 2k-2$, we obtain
\begin{eqnarray*}
c_g^{2k} g^{2k-2}\dot R_gh&=&  (2k-2)!\sum_{r=1}^{2k-2} C^{2k}_r\prod_{i=0}^{r-1}(n-2-i)\frac{g^{2k-2-r}}{(2k-2-r)!}c^{2k-r}_g\dot R_gh\\
&=&  (2k-2)! C^{2k}_{2k-2}\prod_{i=0}^{2k-3}(n-2-i)c^2_g\dot R_gh\\
&=& k(n-2)C_{n,k}\,\,  c^2_g\dot R_gh,
\end{eqnarray*}
that is,
\begin{eqnarray} \label{C}
c_g^{2k} g^{2k-2}\dot R_gh= k(n-2)C_{n,k}\,\,  c^2_g\dot R_gh
\end{eqnarray}
Hence, inserting (\ref{A}), (\ref{B}) and (\ref{C}) into (\ref{svar}), we get
\begin{eqnarray*}
\dot\Sc^{(2k)}_gh &=& \dfrac{C_{n,k}\mu_k}{(2k-1)!}\Big\{-(k-1)\kappa_g \tr_gh
-(n-2k)\langle\R_g^{(2)},h\rangle+\\
&&\quad +\dfrac{k(n-2)}{2} c^2_g\dot R_gh\Big\},
\end{eqnarray*}
which in view of (\ref{rr}) reduces to
\begin{equation}\label{quasi}
\dot\Sc^{(2k)}_gh = D_{n,k}\mu_k\left( \Delta_g \tr_gh+\delta_g \delta_g h+ \langle T_g,h\rangle\right),
\end{equation}
were
\begin{equation}\label{quasi2}
T_g=\frac{1}{k(n-2)}\left((kn+2k-2n)\R_g^{(2)} -2(k-1)\kappa_g g\right)
\end{equation}
and
\begin{equation}\label{quasi3}
D_{n,k}=\dfrac{k^2(n-2)C_{n,k}}{(2k)!}.
\end{equation}

\begin{proposition}\label{linescar}
If $(X,g)\in\Hc_{n,k}$  then
\begin{equation}\label{linescar2}
\dot\Sc^{(2k)}_gh = D_{n,k}\mu_k\left( \Delta_g \tr_gh+\delta_g \delta_g h -\dfrac{\kappa_g}n  \tr_gh \right).
\end{equation}
\end{proposition}
\begin{proof}
It suffices to observe that, due to the Proposition \ref{equivform}, $g$ is Einstein, that is,  $\Rc^{(2)}_g=\frac{\kappa_g}{n}g$.
\end{proof}

\begin{remark}\label{reduces}
The point of (\ref{linescar2}) is that, computed at manifolds in $\Hc_{n,k}$, the linearization of the $2k$-Gauss-Bonnet curvature has, up to a constant, the same expression as the linearization of the scalar curvature ($k=1$); see (\ref{curvscarvar}).
\end{remark}

We shall use  this formula for infinitesimal conformal deformations, namely, $h=fg$, $f\in \Dc(X)$; see (\ref{tangent}). In this situation, $\delta_g(fg)=-df$, hence $\delta_g\delta_g(fg)=-\Delta_gf$, and the next corollary is immediate.

\begin{corollary}\label{cro}
If $(X,g)\in\Hc_{n,k}$ then
\begin{equation}\label{linescfin}
\dot\Sc^{(2k)}_g(fg)=D_{n,k}'\Lc_g,
\end{equation}
where
\begin{equation}\label{linescar4}
\Lc_g=\Delta_g-\frac{\kappa_g}{n-1}
\end{equation}
and
\begin{equation}\label{linescar3}
D_{n,k}'=(n-1)D_{n,k}.
\end{equation}
\end{corollary}

The following fact about the operator $\Lc_g$, for $(X,g)$ in the subclass $\Hc'_{n,k}$, will play a key role in our analysis.

\begin{proposition}\label{specbotton}
If $(X,g)\in\Hc'_{n,k}$ then either $\ker \Lc_g$ is trivial or is formed by constant functions.
\end{proposition}
\begin{proof}
The result is obvious if $\kappa_g\leq 0$ since $\Delta_g$ is nonnegative. On the other hand, if $\kappa_g>0$,
a result due to Lichnerowicz and Obata \cite{BGM} implies, from the fact that $(X,g)$ is Einstein, that the first eigenvalue of $\Delta_g$ is greater than or equal to $\kappa_g/(n-1)$, with the equality holding if and only if $(X,g)$ is a round sphere.
\end{proof}

\section{The proof of  Theorem \ref{main5}}\label{demmain5}

Let $(X,\gm)\in \Hc'_{n,k}$, so that, in particular, $g_0$ satisfies $R_{g_0}^{k-1}=\mu_k{g_0}^{2k-2}$, $\mu_k\neq 0$. Applying a homothety to $g_0$ we may assume that $\gm\in\Mc_1(X)$, the space of unit volume metrics. Hence, we will prove Theorem \ref{main5} under the condition $\nu=1$.

In the following, $H_{\gm}^r(\Uc)$ will denote the standard Sobolev construction applied to an open subset $\Uc$ of sections of a vector bundle over $X$, so that, for instance,  $H_{\gm}^r(\Mc(X))$ is the Hilbert manifold, modeled on $H^r_\gm(\Cc^1(X))$, of metrics with derivatives up to order $r$ defined almost everywhere and square integrable (with respect to   $\gm$).

Choose $r>\frac{n}{2}+4$ and define $\Bc_r:H^r_\gm(\Mc(X))\to H^{r-4}_\gm(\Dc_{\bullet}(X))$ by
\[
\Bc_r(g)=\Delta_{g}\Sc^{(2k)}_{g}-\int_X \Delta_{g}\Sc^{(2k)}_{g}\nu_\gm,
\]
were
\[
\Dc_{\bullet}(X)=\left\{ \rho\in \Dc(X);\int_X \rho\,\nu_\gm=0\right\}.
\]
Since $g\in H^r_\gm(\Mc(X))$ implies $R_{g}\in H^{r-2}_\gm(\Cc^2(X))$, $\Bc_r$ is well-defined and smooth due to the local expression for $\Sc_{g}^{(2k)}$, namely,
\[
\Sc_{g}^{(2k)}= e_{n,k}\delta^{i_1i_2\ldots i_{2k-1}i_{2k}}_{j_1j_2\ldots j_{2k-1}j_{2k}}R_{i_1i_2}^{j_1j_2}\ldots R_{i_{2k-1}i_{2k}}^{j_{2k-1}j_{2k}},
\]
which follows by contracting (\ref{local}), and the fact that, for $r-2>n/2+2>n/2$, the Sobolev space $H^{r-2}_\gm$ is a Banach algebra with respect to pointwise multiplication \cite{MS}.

\begin{lemma}\label{lemma0}
There exists a neighborhood, say $V^r$, of $\gm$ in $H_{\gm}^r(\Mc_1^{(2k)}(X))$ which is a smooth submanifold of $H_{\gm}^{r}(\Mc(X))$ with $T_{\gm}V^r=\ker \dot\Bc_r(\gm)$.
\end{lemma}

\begin{proof}
Note that  $\dot\Delta_\gm (h)\Sc^{(2k)}_{\gm}=0$ for $h\in\Cc^1(X)$, due to the fact that  $\Sc^{(2k)}_{\gm}$ is constant. Hence we obtain from  (\ref{linescar2}) that
\begin{eqnarray}
\dot\Bc_r(\gm)(h) & = & \Delta_\gm\dot\Sc^{(2k)}_{\gm}(h)\nonumber\\
\label{nost} & = & D_{n,k}\mu_k\Delta_\gm\left(\Delta_\gm \mathsf {tr}_\gm h+\delta_\gm \delta_\gm h-\frac{\kappa_{g_0}}{n}{{\rm tr}}_\gm h\right).
\end{eqnarray}
Now, let $h=f\gm$ for $f\in H_{\gm}^{r}(\Dc_\bullet(X))$, so that, by (\ref{linescfin}),
\[
\dot\Bc_r(\gm)(f\gm)=D'_{n,k}\mu_k\Delta_\gm\Lc_\gm f.
\]
Using Proposition \ref{specbotton} it is easy to see that $\dot\Bc_r(\gm)|_{H_{\gm}^{r}(\Dc_\bullet(X))g_0}$ is injective and the Fredholm alternative then implies that
\[
\dot\Bc_r(\gm):H_{\gm}^r(\Cc^1(X))\to H_{\gm}^{r-4}(\Dc_\bullet(X))
\]
is surjective. The lemma now is a straightforward consequence of the Implicit Function Theorem and the fact that
$\Bc^{-1}_r(0)=\Mc_1^{(2k)}(X)$.
\end{proof}

\begin{lemma}\label{lemma}
If $\Dc_\bullet^+(X)=\Dc_\bullet(X)\cap\Dc^+(X)$ and
\[
\xi^r:H_{\gm}^r(\Dc^+_\bullet(X))\times V^r\to H_{\gm}^r(\Mc(X))
\]
is the smooth map given by $\xi^r(f,g)=fg$ then $d\xi^r_{(1,\gm)}$ is an isomorphism.
\end{lemma}

\begin{proof}
 If $d\xi^r_{(1,\gm)}(\phi,h)=h+\phi \gm=0$ then $h=-\phi \gm\in \ker \dot\Bc_r(\gm)$ so that $\Lc_\gm\Delta_\gm\phi=0$. Thus $\Delta_\gm\phi=0$ by the Proposition \ref{specbotton} and $\phi$ is constant. But $\int_X\phi\,\nu_\gm=0$ because $V^r\subset H_{\gm}^r(\Mc_1(X))$ and thus $\phi=0$, implying that $h=0$. This shows the injectivity of $d\xi_{(1,\gm)}$.

For the surjectivity note that the decomposition
\[
{\rm Im}\,d\xi^r_{(1,\gm)}=T_\gm V^r\oplus H_{\gm}^r(\Dc_\bullet(X))\gm
\]
already shows that ${\rm Im}\,d\xi^r_{(1,\gm)}$ is closed in $H_{\gm}^r(\Cc^1(X))$. Now assume by contradiction the existence of  $h\neq 0$ in $H_{\gm}^r(\Cc^1(X))$ orthogonal both to $T_\gm V^r$ and $H_{\gm}^r(\Dc_\bullet(X)\gm)$. It follows from (\ref{nost}) that $\dot\Bc_r(\gm)$ has surjective symbol and since $T_{\gm}V^r=\ker \dot\Bc_r(\gm)$ one has the decomposition \cite{E}:
\[
H_{\gm}^r(\Cc^1(X))=\mathbb{R}\gm\oplus T_\gm V^r\oplus {\rm Im}\,\dot\Bc_r(\gm)^*,
\]
where $\dot\Bc_r(\gm)^*$ is the $L^2$ adjoint of $\dot\Bc_r(\gm)$. This allows us to write $h=\dot\Bc_r(\gm)^*(\varphi)$, that is,
\[
h=D_{n,k}\mu_k\left((\Delta^2_\gm\varphi)\gm+\nabla^2\Delta_\gm\varphi-\frac{\kappa_{g_0}}{n} (\Delta_\gm\varphi) \gm\right),
\]
and taking traces,
\[
{\rm tr}_\gm h=D'_{n,k}\mu_k\Lc_\gm\Delta_\gm\varphi.
\]
But, $\int_X{\rm tr}_\gm h\,\nu_\gm=0$ because $h$ is orthogonal to  $H_{\gm}^r(\Dc_\bullet(X))\gm$, so  if we use  Proposition \ref{specbotton}, $\int_X\Delta_\gm\varphi\,\nu_\gm=0$ and the variational characterization of the first eigenvalue $\lambda_1(\Delta_\gm)$,
we get
\[
\frac{\kappa_g}{n-1}<\lambda_1(\Delta_\gm)\leq \frac{\int_X|\nabla\Delta_{\gm}\varphi|^2\nu_{\gm}}
{\int_X |\Delta_{\gm}\varphi|^2\nu_{\gm}}= \frac{\kappa_g}{n-1},
\]
a contradiction unless
$\Delta_\gm\varphi=0$, that is,  $\varphi$ is constant and therefore $h=0$. \end{proof}

With Lemmas \ref{lemma0} and \ref{lemma} at hand the proof of Theorem  \ref{main5} is immediate, following essentially from the fact that objects in the ILH category are defined as inverse limits of objects in the  $H_{\gm}^r$ category as $r\to +\infty$.  Therefore, we shall omit the details and refer instead to \cite{K}; see his  proof of Theorem 2.5.


\bibliographystyle{amsplain}

\end{document}